\documentclass[a4paper] {article}
[12pt]
\makeatletter
\renewcommand*\l@section{\@dottedtocline{1}{1.5em}{2.3em}}
\makeatother
\usepackage{amsfonts}
\usepackage{amssymb}
\usepackage[colorlinks=true,
            linkcolor=blue,
            anchorcolor=blue,  %%ÐÞ¸Ä´Ë´¦ÎªÄãÏëÒªµÄÑÕÉ«
            citecolor=blue,        %%ÐÞ¸Ä´Ë´¦ÎªÄãÏëÒªµÄÑÕÉ«£¬ÀýÈçÐÞ¸ÄblueÎªred
            ]{hyperref}
\usepackage[T1]{fontenc}

\usepackage{tikz}\usepackage{float}%\usepackage{pgfplots}
\usetikzlibrary{calc}

\usepackage{CJK}
\usepackage{amsmath}

\usepackage{amsfonts}
\usepackage{amssymb}
\usepackage{amsthm}
\usepackage{amssymb}
\usepackage{enumerate}
\usepackage[calc]{picture}\usepackage{graphicx}\usepackage{subfigure}\usepackage{cite}\usepackage{booktabs}
\usepackage[all,cmtip]{xy}

\usepackage{eqlist}\usepackage{mathrsfs}

\usepackage{color}
\usepackage{abstract}

\newcommand{\midarrow}{\tikz \draw[-triangle 90] (0,0) -- +(.1,0);}%draw digraph
\usetikzlibrary{arrows}%draw digraph
\usepackage[top=3.0 cm, bottom=3.0cm, left=3.5cm, right= 3.5cm]{geometry}

\usepackage{cite}

\theoremstyle{plain}
\newtheorem{theorem}{Theorem}[section]
\newtheorem{proposition}[theorem]{Proposition}
\newtheorem{lemma}[theorem]{Lemma}
\newtheorem{corollary}[theorem]{Corollary}

\theoremstyle{definition}
\newtheorem{definition}{Definition}[section]
\newtheorem{example}{Example}[section]

\theoremstyle{remark}%{myrem}
 \newtheorem{remark}{Remark}[section]
\newcommand{\im}{{\mathrm{Im}\hspace{0.1em}}}
\usepackage{etoolbox}
\numberwithin{equation}{section}
\numberwithin{theorem}{section}

\begin{document}
    \begin{CJK*}{GBK}{kai}
    \CJKtilde

\begin{center}
{\Large {\textbf {The algebraic stability for persistent Laplacians %\\
%The Laplacian tree and its stability
}}}
 \vspace{0.58cm}

Jian Liu, Jingyan Li, Jie Wu*

\bigskip

\bigskip

    \parbox{24cc}{{\small
{\textbf{Abstract}.}
The stability of topological persistence is one of the fundamental issues in topological data analysis. Numerous methods have been proposed to address the stability of persistent modules or persistence diagrams. Recently, the concept of persistent Laplacians has emerged as a novel approach to topological persistence, attracting significant attention and finding applications in various fields. In this paper, we investigate the stability of persistent Laplacians.
We introduce the notion of ``Laplacian trees'', which captures the collection of persistent Laplacians that persist from a given parameter. To formalize our study, we construct the category of Laplacian trees and establish an algebraic stability theorem for persistent Laplacian trees. Notably, our stability theorem is applied to the real-valued functions on simplicial complexes and digraphs.

}}
\end{center}

\vspace{1cc}
%\tableofcontents



\footnotetext[3] {* corresponding author. }
%\tableofcontents
\section{Introduction}
Persistent homology, a relatively recent method for extracting topological features and recognizing geometric shapes, has evolved into a well-established theory \cite{zomorodian2004computing,carlsson2004persistence,carlsson2007theory,carlsson2010zigzag,edelsbrunner2000topological}.
The stability for topological persistence is one of the fundamental issues in topological data analysis. In \cite{cohen2005stability}, D. Cohen-Steiner, H. Edelsbrunner, and J. Harer began to study the stability of persistence diagram with respect to the real-valued functions on a given topological space.
The notion of ``$\varepsilon$-interleaved'' was subsequently introduced by F. Chazal, D. Cohen-Steiner, and M. Glisse et al to  describe the algebraic stability theory for persistence diagrams of functions defined on different spaces \cite{chazal2009proximity}.
In 2013, the authors obtained the induced matching theorem, which provides a generalization and alternative proof of the algebraic stability theorem for persistence diagrams \cite{bauer2015induced}.
In \cite{bubenik2014categorification}, the authors provided a categorification of the algebraic stability theorem and studied the category of $\varepsilon$-interleavings.
In 2015, H. Edelsbrunner, G. Jab{\l}o\'{n}ski, and M. Mrozek explored the persistent homology of self-maps and presented a stability theorem for self-maps in terms of interleaving distance. \cite{edelsbrunner2015persistent}.
Then U. Bauer and M. Lesnick provided simplified reformulations of the stability theorem and the induced matching theorem in terms of the category of diagrams \cite{bauer2020persistence}.
Recently, A. J. Blumberg and M. Lesnick investigated the stability of 2-parameter persistent homology. \cite{blumberg2022stability}.
Stability theorems have also been considered for other objects, such as the stability of persistent homology for hypergraphs \cite{ren2020stability}, the stability of persistent path homology for networks \cite{chowdhury2018persistent}, and the stability of persistence diagrams over lattices \cite{mccleary2022edit}.

The classical de Rham-Hodge theory establishes an isomorphism between the de Rham cohomology and the space of harmonic forms for an oriented Riemannian manifold without boundary. In a similar vein, the formulation of discrete de Rham-Hodge theory in the context of topological persistence leads to analogous conclusions.
Since 2019, G.-W. Wei and his coauthors have introduced and explored the concept of persistent Laplacians, revealing their potential for capturing additional topological information and geometric features \cite{chen2021evolutionary,wang2020persistent,meng2021persistent}. These studies have demonstrated the richness of persistent Laplacians and their applications. In a recent work by F. Memoli, Z. Wan, and Y. Wang \cite{memoli2022persistent}, a theoretical investigation and efficient algorithms for persistent Laplacians were presented. Notably, they examined the stability theorem for up-persistent eigenvalues in terms of the interleaving distance.
However, the stability of Laplacians remains an implicit aspect to date. One challenge in studying persistent Laplacians lies in describing their underlying structure. While persistence modules or persistence diagrams can encapsulate all the information of persistent homology, the collection of persistent Laplacians, unlike other algebraic objects such as vector spaces, poses difficulties in achieving efficient representation. Additionally, the algebraic stability theorem for persistence diagrams heavily relies on categorical constructions of persistence objects. When dealing with persistent Laplacians, a significant challenge is to identify a suitable categorical framework that can facilitate the development of an algebraic stability theorem.

In this paper, our objective is to establish an algebraic stability theorem for persistent Laplacians. We initiate this by categorizing the persistent algebraic Laplacians within the category $\mathbf{DGI}$ of differential graded inner product spaces. Let $(\mathbb{R},\leq)$ be the category with objects represented by real numbers and morphisms denoted by $a\to b$ for any real numbers $a\leq b$. A \emph{persistence object on a category $\mathfrak{C}$} is defined as a functor $\mathcal{F}:(\mathbb{R},\leq)\to \mathfrak{C}$ from the category $(\mathbb{R},\leq)$ to the category $\mathfrak{C}$. Consider $\mathcal{S}:(\mathbb{R},\leq)\to \mathbf{DGI}$, where $\mathbf{DGI}$ represents the category of differential graded inner product spaces. Within this context, there are two functors from the category $\mathbf{DGI}$ of differential graded inner product spaces to the category $\mathbf{Inn}$ of inner product spaces, the homology functor $H:\mathbf{DGI}\to \mathbf{Inn}$ and the harmonic space functor $h:\mathbf{DGI}\to \mathbf{Inn}$ (Section \ref{section:Laplacians}).
\begin{theorem}\label{introduction:theorem1}
There is a natural isomorphism $h\mathcal{S}\Rightarrow H\mathcal{S}$ of functors from $(\mathbb{R},\leq)$ to $\mathbf{Inn}$.
\end{theorem}
The above theorem establishes the isomorphism between persistent homology and persistent harmonic space as persistence modules, meaning they have isomorphic persistence diagrams. Specifically, for any real numbers $a\leq b$, the \emph{$(a,b)$-persistent harmonic space} is defined as
\begin{equation*}
\mathcal{H}^{a,b}(\mathcal{S})=hf_{a\to b}h(\mathcal{S}{a}),
\end{equation*}
which is naturally isomorphic to the $(a,b)$-persistent homology $H^{a,b}(\mathcal{S})$ with respect to the morphism $a\to b$. Here, the morphism $f{a\to b}:\mathcal{S}{a}\to \mathcal{S}{b}$ is induced by the morphism $a\to b$.
The \emph{$(a,b)$-persistent Laplacian} \cite{memoli2022persistent} is defined by
\begin{equation*}
  \Delta^{a,b}_{\mathcal{S}}=f^{\ast}_{a\to b}d\iota\iota^{\ast}d^{\ast}f_{a\to b}+d^{\ast}d.
\end{equation*}
Here, $\iota:\{x\in\mathcal{S}_{b}|dx\in f_{a\to b}(\mathcal{S}_{a})\}\hookrightarrow \mathcal{S}_{b}$ is an inclusion.
We present the persistent Hodge decomposition theorem for the persistence differential graded inner product space (Theorem \ref{theorem:persistencedecomposition}).
\begin{theorem}\label{introduction:theorem2}
Let $\mathcal{S}:(\mathbb{R},\leq)\to \mathbf{DGI}$ be a persistence differential graded inner product space. For any $a\leq b$, we have a direct sum decomposition of inner product spaces
\begin{equation*}
  \mathcal{S}_{a}= \ker \Delta^{a,b}_{\mathcal{S}}\oplus \im f_{a\to b}^{\ast}d\iota\iota^{\ast}d^{\ast}f_{a\to b}\oplus \im d^{\ast}d,
\end{equation*}
where $\ker \Delta^{a,b}_{\mathcal{S}}\cong\mathcal{H}^{a,b}(\mathcal{S})$.
\end{theorem}
Theorems \ref{introduction:theorem1} and \ref{introduction:theorem2} are essentially based on the idea that the nullity of the $q$-th persistent Laplacian equals the $q$-th persistent Betti number in \cite{wang2020persistent,memoli2022persistent} and reformed in a more categorical framework in this paper.

In this paper, our main focus is on providing a categorization of algebraic Laplacians using Laplacian trees and investigating the stability of persistence Laplacian trees.
The Laplacian tree, as a collection of Laplacians on $V$, is defined as a pair $(V,A)$, where $V$ is a differential graded inner space and $A$ is a collection of linear transformations on $V$ such that each linear transformation is of the form
\begin{equation*}
\Delta_{f}=f^{\ast}d\iota\iota^{\ast}d^{\ast}f+d^{\ast}d,
\end{equation*}
where $f:V\to W$ is a morphism of differential graded inner spaces. Here, $\iota:{x\in W|dx\in f(V)}\hookrightarrow W$ is an inclusion.
In Section \ref{section:laplaciancategory}, we list some properties of Laplacian trees and show that the Laplacian tree contains the information of persistent homology.
The categorization of Laplacian trees serves as a fundamental framework for investigating the stability of Laplacians. As mentioned earlier, interleaving plays a crucial role in studying the algebraic stability of persistence modules.
Let $\mathcal{F},\mathcal{G}:(\mathbb{R},\leq)\to \mathfrak{C}$ be two persistence objects in $\mathfrak{C}^{\mathbb{R}}$. The \emph{interleaving distance} \cite{chazal2009proximity} between $\mathcal{F}$ and $\mathcal{G}$ is
\begin{equation*}
  d_{I}(\mathcal{F},\mathcal{G})=\inf\{\varepsilon\geq 0|\text{$\mathcal{F}$ and $\mathcal{G}$ are $\varepsilon$-interleaved}\}.
\end{equation*}
In this paper, for a given persistence differential graded inner product space $\mathcal{S}$, we construct a persistence Laplacian tree $\mathcal{L}_{\mathcal{S}}$. The construction is functorial. Moreover, we have the algebraic stability of persistence Laplacian trees (Theorem \ref{theorem:main2}).
\begin{theorem}[\textbf{The algebraic stability theorem}]\label{introduction:interleaving}
Let $\mathcal{S},\mathcal{T}$ be two persistence differential graded inner product spaces. Then the persistence Laplacian trees $\mathcal{L}_{\mathcal{S}}$ and $\mathcal{L}_{\mathcal{T}}$ are $\varepsilon$-interleaved if and only if $\mathcal{S}$ and $\mathcal{T}$ are $\varepsilon$-interleaved.
\end{theorem}
A straightforward corollary of Theorem \ref{introduction:interleaving} is as follows (Corollary \ref{corollary:distance}).
\begin{corollary}
Let $\mathcal{S},\mathcal{T}$ be two persistence differential graded inner product spaces. Then
\begin{equation*}
  d_{I}(\mathcal{L}_{\mathcal{S}},\mathcal{L}_{\mathcal{T}})=d_{I}(\mathcal{S},\mathcal{T}).
\end{equation*}
\end{corollary}
As applications of the algebraic stability theorem for persistence Laplacian trees, we consider real-valued functions on simplicial complexes and digraphs. A real-valued function on a simplicial complex $f:K\to \mathbb{R}$ is non-decreasing if $f(\sigma)\leq f(\tau)$ whenever $\sigma$ is a face of $\tau$.
Let $f$ and $g$ be two non-decreasing real-valued functions on a simplicial complex $K$. Recall that $\|f-g\|_{\infty}=\sup\limits_{\sigma\in K}|f(\sigma)-g(\sigma)|$ is the usual $L_{\infty}$-distance of real-valued functions on $K$. We obtain the following theorem (Theorem \ref{theorem:complex}).
\begin{theorem}
Let $f$ and $g$ be two non-decreasing real-valued functions on a simplicial complex $K$. Then
\begin{equation*}
d_{I}(\mathcal{L}^{f},\mathcal{L}^{g})\leq \|f-g\|_{\infty},
\end{equation*}
where $\mathcal{L}^{f}$ and $\mathcal{L}^{g}$ are the persistence Laplacian trees of $f$ and $g$ on $K$, respectively.
\end{theorem}
The GLMY theory, introduced in \cite{grigor2012homologies,grigor2020path,grigor2017homologies}, provides a well-established topological invariant for digraphs and has significant potential for applications.
For a weighted digraph $(G,f)$, where $G$ is a digraph and $f$ is a real-valued function on its edge set, we have persistence inner product spaces $\mathcal{H}^{f}(G)$ and $H^{f}(G)$, which correspond to persistent harmonic space and persistent homology, respectively, as discussed in Section \ref{section:digraph}.
The following theorem (Theorem \ref{theorem:2digraph}) presents an algebraic stability theorem for digraphs.
\begin{theorem}
Let $(G_{1},f)$ and $(G_{2},g)$ be weighted digraphs. Then
\begin{equation*}
  d_{I}(\mathcal{H}^{f}(G_{1}),\mathcal{H}^{g}(G_{2}))=d_{I}(H^{f}(G_{1}),H^{g}(G_{2}))\leq d_{I}(\mathcal{L}^{(G_{1},f)},\mathcal{L}^{(G_{2},g)})\leq d_{I}((G_{1},f),(G_{2},g)).
\end{equation*}
\end{theorem}
\begin{corollary}
Let $(G,f)$ and $(G,g)$ be weighted digraphs. Then
\begin{equation*}
  d_{I}(\mathcal{H}^{f}(G),\mathcal{H}^{g}(G))=d_{I}(H^{f}(G),H^{g}(G))\leq d_{I}(\mathcal{L}^{f},\mathcal{L}^{g})\leq \|f-g\|_{\infty},
\end{equation*}
where $\mathcal{L}^{f}$ and $\mathcal{L}^{g}$ are the persistence Laplacian trees of functions $f$ and $g$ on the digraph $G$, respectively. Here, $\|f-g\|_{\infty}=\sup\limits_{e\in E}|f(e)-g(e)|$.
\end{corollary}
This paper is organized as follows. In the next section, we introduce the notations concerning persistence objects and $\varepsilon$-interleavings. Section \ref{section:Laplacians} focuses on the study of persistent Laplacians on differential graded inner spaces. In Section \ref{section:laplaciancategory}, we present the construction of the category of Laplacian trees. The algebraic stability theorem for persistence Laplacian trees is shown in Section \ref{section:main}. In the last section, as an application, we consider the persistent Laplacians and the algebraic stability of real-valued functions on digraphs.

\section{Preliminaries}
This section recalls the notations and definitions established on persistent homology. Categorical language is accepted to help build the algebraic stability theorem.
The idea of the $\varepsilon$-interleaving was first introduced by F. Chazal, D. Cohen-Steiner, and M. Glisse et al to study the stability for persistence modules or persistence diagrams. The developed method based on interleavings generalizes many classical results concerning the stability for persistence diagrams. To study the stability for persistent Laplacians, we inherit the theoretical framework on the $\varepsilon$-interleaving.

\subsection{Persistence objects}
Let $\mathfrak{C}$ be a category. A \emph{persistence object} on $\mathfrak{C}$ is a functor $\mathcal{F}:(\mathbb{R},\leq)\to \mathfrak{C}$ from the category $(\mathbb{R},\leq)$ to the category $\mathfrak{C}$.
Some familiar concepts in persistent homology theory can be restated in the terms of persistence objects.

\begin{example}
Let $\mathbf{Simp}$ be the category of simplicial complexes. A functor $\mathcal{F}:(\mathbb{R},\leq)\to \mathbf{Simp}$ is a persistence simplicial complex or a filtration of simplicial complexes. Let $X$ be a data set, the \v{C}ech complex \cite{carlsson2009topology} and the Vietoris-Rips complex \cite{vietoris1927hoheren} on $X$ are persistence simplicial complexes parameterized by the distance.
\end{example}
\begin{example}
Let $\mathbf{Vec}_{\mathbb{K}}$ be the category of vector space over a field $\mathbb{K}$. A functor $\mathcal{F}:(\mathbb{R},\leq)\to \mathbf{Vec}_{\mathbb{K}}$ is a persistence ($\mathbb{K}$-)module or a persistence $\mathbb{K}$-vector space.
\end{example}
\begin{example}
Let $\mathbf{Chain}_{\mathbb{K}}$ be the category of chain complexes over a field $\mathbb{K}$. A functor $\mathcal{F}:(\mathbb{R},\leq)\to \mathbf{Chain}_{\mathbb{K}}$ is a persistence chain complex.
\end{example}

\begin{remark}
Let $(X,\leq)$ be a poset. We can regard $X$ as a category with elements in $X$ as objects and the binary relations in $\leq$ as morphisms. An \emph{$(X,\leq)$-persistence} object on $\mathfrak{C}$ is a functor $\mathcal{F}:(X,\leq)\to \mathfrak{C}$ from the category $(X,\leq)$ to the category $\mathfrak{C}$. For example, the posets $(\mathbb{Z},\leq)$ and $(\mathbb{R}^{n},\leq)$ lead to the usual discrete persistence and multi-parameter persistence, respectively.
\end{remark}

All the persistence objects on $\mathfrak{C}$ consist of a category $\mathfrak{C}^{\mathbb{R}}$. More precisely, $\mathfrak{C}^{\mathbb{R}}$ is a category with
\begin{itemize}
  \item objects: all the functors $\mathcal{F}:(\mathbb{R},\leq)\to \mathfrak{C}$;
  \item morphisms: all the natural transformation $T: \mathcal{F}\Rightarrow \mathcal{G}$ between functors.
\end{itemize}
According to the terminology in \cite{bubenik2014categorification}, we can also call $\mathfrak{C}^{\mathbb{R}}$ the category of $(\mathbb{R},\leq)$-indexed diagrams in the category $\mathfrak{C}$.
In \cite{edelsbrunner2015persistent,bauer2020persistence}, the authors showed that the barcodes can be equivalently described by the persistence matchings.
Moreover, they give an equivalence functor $\mathbf{Barc}\to \mathbf{Mch}^{\mathbb{R}}$ from the category of barcodes to the category of persistence matchings. Thus one has ``equivalences'' among the category of (finite-type) persistence modules, the category of (finite) barcodes, and the category of (finite) matchings.
\begin{equation*}
  \xymatrix{(\mathbf{Vec}_{\mathbb{K}})^{\mathbb{R}}\ar@{<=>}[r]&\mathbf{Barc}\ar@{<=>}[r]&\mathbf{Mch}^{\mathbb{R}}}
\end{equation*}
In Section \ref{subsection:interleaving}, we will review the definition of the $\varepsilon$-interleaving between persistence objects, which plays an important role in the algebraic stability theorem.

\subsection{Interleavings of persistence objects}\label{subsection:interleaving}

%Let $T_{x}:(\mathbb{R},\leq)\to (\mathbb{R},\leq)$ be a functor given by $T_{x}(a)=a+x$. Let $\eta_{x}:\mathrm{id}_{(\mathbb{R},\leq)}\to T_{x}$ be a natural transformation. Thus we have the following commutative diagram.
%\begin{equation*}
%  \xymatrix{
%  a\ar@{->}[r]\ar@{->}[d]_{f_{a,b}}&a+x\ar@{->}[d]^{T_{x}f_{a,b}=\eta_{x}(\mathrm{id}_{(\mathbb{R},\leq)})(f_{a,b})}\\
%  b\ar@{->}[r]&b+x
%  }
%\end{equation*}
%Here, $f_{a,b}:a\to b$ is a morphism in the category $(\mathbb{R},\leq)$. It can be verified that $\eta_{x}\eta_{y}=\eta_{x+y}$.
%
%
%\begin{definition}
%Let $\mathcal{F},\mathcal{G}:(\mathbb{R},\leq)\to \mathfrak{C}$ be two persistence objects. An \emph{$\varepsilon$-interleaving between $\mathcal{F}$ and $\mathcal{G}$} consists of natural transformations $\varphi:\mathcal{F}\Rightarrow \mathcal{G}T_{\varepsilon}$ and $\psi:\mathcal{G}\Rightarrow \mathcal{F}T_{\varepsilon}$ such that
%\begin{equation*}
%  (\psi T_{\varepsilon})\varphi=\mathcal{F}\circ\eta_{2\varepsilon},\quad (\varphi T_{\varepsilon})\psi=\mathcal{G}\circ\eta_{2\varepsilon}.
%\end{equation*}
%We say that $\mathcal{F}$ and $\mathcal{G}$ are $\varepsilon$-interleaved.
%\end{definition}

Let $T_{x}:(\mathbb{R},\leq)\to (\mathbb{R},\leq)$ be a functor given by $T_{x}(a)=a+x$ for $a,x\in \mathbb{R}$. Then $T_{x}$ induces a functor $\Sigma^{x}:\mathfrak{C}^{\mathbb{R}}\to \mathfrak{C}^{\mathbb{R}}$ of the categories of persistence objects given by $(\Sigma^{x}\mathcal{F})(a)=\mathcal{F}(a+x)$. Note that $\Sigma^{x}|_{\mathcal{F}}:\mathcal{F}\to \Sigma^{x}\mathcal{F}$ is a morphism in the category $\mathfrak{C}^{\mathbb{R}}$ of persistence objects. %If there is no ambiguity, we denote $\Sigma^{x}|_{\mathcal{F}}$ by $\Sigma^{x}$ for convenience.

\begin{definition}
Let $\mathcal{F},\mathcal{G}:(\mathbb{R},\leq)\to \mathfrak{C}$ be two persistence objects. An \emph{$\varepsilon$-interleaving between $\mathcal{F}$ and $\mathcal{G}$} consists of two morphisms $\phi:\mathcal{F}\to \Sigma^{\varepsilon}\mathcal{G}$ and $\psi:\mathcal{G}\to \Sigma^{\varepsilon}\mathcal{F}$ in $\mathfrak{C}^{\mathbb{R}}$ such that $(\Sigma^{\varepsilon}\psi)\phi=\Sigma^{2\varepsilon}|_{\mathcal{F}}$ and $(\Sigma^{\varepsilon}\phi)\psi=\Sigma^{2\varepsilon}|_{\mathcal{G}}$. We say that $\mathcal{F}$ and $\mathcal{G}$ are \emph{$\varepsilon$-interleaved}.
\end{definition}

\begin{remark}
The $\varepsilon$-interleaving between $\mathcal{F}$ and $\mathcal{G}$ defined above can be represented by the following commutative diagrams.
\begin{equation*}
  \xymatrix@=0.6cm{
  &\Sigma^{\varepsilon}\mathcal{G}\ar@{->}[rd]^{\Sigma^{\varepsilon}\psi}&\\
  \mathcal{F}\ar@{->}[ru]^{\phi}\ar@{->}[rr]^{\Sigma^{2\varepsilon}|_{\mathcal{F}}}&&\Sigma^{2\varepsilon}\mathcal{F}
  }\qquad \qquad
  \xymatrix@=0.6cm{
  &\Sigma^{\varepsilon}\mathcal{F}\ar@{->}[rd]^{\Sigma^{\varepsilon}\phi}&\\
  \mathcal{G}\ar@{->}[ru]^{\psi}\ar@{->}[rr]^{\Sigma^{2\varepsilon}|_{\mathcal{G}}}&&\Sigma^{2\varepsilon}\mathcal{G}
  }
\end{equation*}
If $\varepsilon=0$, the above equations reduce to $\psi\phi=\mathrm{id}|_{\mathcal{F}}$ and $\phi\psi=\mathrm{id}|_{\mathcal{G}}$. This shows that if $\mathcal{F}$ and $\mathcal{G}$ are $0$-interleaved, then they are isomorphic.
\end{remark}

We recall the definition of interleaving distance.
\begin{definition}\label{definition:distance}
Let $\mathcal{F},\mathcal{G}:(\mathbb{R},\leq)\to \mathfrak{C}$ be two persistence objects. The \emph{interleaving distance between $\mathcal{F}$ and $\mathcal{G}$} is
\begin{equation*}
  d_{I}(\mathcal{F},\mathcal{G})=\inf\{\varepsilon\geq 0|\text{$\mathcal{F}$ and $\mathcal{G}$ are $\varepsilon$-interleaved}\}.
\end{equation*}
\end{definition}

If $\mathcal{F},\mathcal{G}$ are $\varepsilon$-interleaved, we have $d(\mathcal{F},\mathcal{G})\leq \varepsilon$. In \cite{bubenik2014categorification}, the function $d_{I}$ was proved to be an extended pseudometric on the class of  $(\mathbb{R},\leq)$-indexed diagrams in $\mathfrak{C}$. At last, recall that there are constructions
\begin{equation*}
  \xymatrix{(\mathbf{Vec}_{\mathbb{K}})^{\mathbb{R}}\ar@{->}[r]^{B}&\mathbf{Barc}\ar@{->}[r]^{E}&\mathbf{Mch}^{\mathbb{R}}}.
\end{equation*}
Here, $E$ is a functor while $B$ is not always a functor. In \cite{bauer2020persistence}, the authors showed that
\begin{equation*}
  d_{I}(\mathcal{F},\mathcal{G})=d_{I}(B(\mathcal{F}),B(\mathcal{G}))=d_{I}(EB(\mathcal{F}),EB(\mathcal{G})).
\end{equation*}
To establish the interleavings for Laplacians, one of the fundamental tasks is to categorize the Laplacians. This involves providing a categorical framework that captures the essential properties and structure of Laplacians.

\section{The persistent Hodge decomposition}\label{section:Laplacians}
In this paper, the ground field considered is assumed to be the real number field $\mathbb{R}$.
Let $(V,d)$ be a differential graded inner product space over $\mathbb{R}$ with an inner product $\langle\cdot,\cdot\rangle$.
From now on, all the differential graded inner product spaces considered are assumed to be finite dimensional. In this section, we assume that the inner product is homogeneous, meaning that $\langle x, y \rangle = 0$ for all $x \in V_i$, $y \in V_j$, and $i \neq j$. Recall that the direct sum of inner product spaces $V$ and $W$, denoted as $V \oplus W$, is an inner product space with the following inner product:
\begin{equation*}
  \langle x,y\rangle_{V\oplus W}=\langle x|_{V},y|_{V}\rangle_{V}+\langle x|_{W},y|_{W}\rangle_{W},
\end{equation*}
where $x|_V$ and $y|_V$ represent the direct summands of $x$ and $y$ on $V$, respectively.
Given a linear operator $T: V \to V$, there exists a unique linear operator $T^{\ast}$ such that $\langle Tx,y\rangle=\langle x,T^{\ast}y\rangle$ for all $x,y\in V$. We call $T^{\ast}$ the \emph{adjoint operator} of $T$. Obviously, we have $T^{\ast\ast}=T$ and $\langle x,Ty\rangle=\langle T^{\ast}x,y\rangle$ for all $x,y\in V$.

\begin{proposition}[Algebraic Hodge decomposition]\label{proposition:algebraic}
Let $V'\stackrel{\phi}{\to} V\stackrel{\psi}{\to} V''$ be linear maps of inner product spaces such that $\psi\phi=0$. Then there is a direct sum decomposition of inner product spaces
\begin{equation*}
  V=\im \phi\phi^{\ast}\oplus \im \psi^{\ast}\psi \oplus \mathcal{H},
\end{equation*}
where $\mathcal{H}=\ker \psi\cap \ker \phi^{\ast}=\ker \Delta$. Here, $\Delta=\phi\phi^{\ast}+ \psi^{\ast}\psi$.
\end{proposition}
\begin{proof}
For any $x,y\in V$, we have
\begin{equation*}
  \langle \phi\phi^{\ast} (x),\psi^{\ast}\psi(y)\rangle=\langle \psi\phi\phi^{\ast} (x),\psi(y)\rangle=0.
\end{equation*}
Thus we obtain $\im \phi\phi^{\ast}\perp \im \psi^{\ast}\psi$. A similar calculation shows that $\im \phi\phi^{\ast}\perp \mathcal{H}$ and $\im \psi^{\ast}\psi\perp \mathcal{H}$. Thus we have $\mathcal{H}\subseteq (\im \phi\phi^{\ast}\oplus \im \psi^{\ast}\psi)^{\perp}$.
To prove the decomposition, we will show $(\im \phi\phi^{\ast}\oplus \im \psi^{\ast}\psi)^{\perp}\subseteq \mathcal{H}$. Note that $(\im \phi\phi^{\ast}\oplus \im \psi^{\ast}\psi)^{\perp}=(\im \phi\phi^{\ast})^{\perp}\cap (\im \psi^{\ast}\psi)^{\perp}$. For $x\in (\im \phi\phi^{\ast})^{\perp}\cap (\im \psi^{\ast}\psi)^{\perp}$, one has $\langle x,\phi\phi^{\ast}(y)\rangle=0$ for all $y\in V$. It follows that $\langle \phi\phi^{\ast}(x),y\rangle=0$ for all $y\in V$. This shows that $\phi\phi^{\ast}(x)=0$, and then we have $x\in \ker \phi^{\ast}$ by the positive definiteness of inner product. Similarly, one has $x\in \ker \psi$. Hence, we obtain that $(\im \phi\phi^{\ast}\oplus \im \psi^{\ast}\psi)^{\perp}\subseteq \mathcal{H}$. Note that $\im \phi\phi^{\ast}$, $\im \psi^{\ast}\psi$ and $\mathcal{H}$ are inner product subspaces of $V$.

It is obvious that $\mathcal{H}\subseteq\ker \Delta$. On the other hand, if $x\in\ker \Delta$, one has
\begin{equation*}
  0=\langle\Delta x,\Delta x\rangle=\langle\phi\phi^{\ast} (x),\phi\phi^{\ast} (x)\rangle+\langle\psi^{\ast}\psi(y),\psi^{\ast}\psi(y)\rangle.
\end{equation*}
By the positive definiteness of the inner product, we have $\phi\phi^{\ast} (x)=\psi^{\ast}\psi(y)=0$. Note that
\begin{equation*}
  \langle \phi^{\ast} (x),\phi^{\ast} (x)\rangle=\langle \phi\phi^{\ast} (x),x\rangle=0.
\end{equation*}
It follows that $\phi^{\ast}(x)=0$. Similarly, one has $\psi(y)=0$. Hence, we obtain $\ker\Delta\subseteq \ker \psi\cap \ker \phi^{\ast}=\mathcal{H}$. It follows that $\mathcal{H}=\ker\Delta$.
\end{proof}

\begin{corollary}\label{corollary:decomposition}
Let $(V,d)$ be a differential graded inner product space. Then there is a direct sum decomposition of inner product spaces
\begin{equation*}
  V=\im d \oplus \im d^{\ast} \oplus \mathcal{H},
\end{equation*}
where $\mathcal{H}=\ker d\cap \ker d^{\ast}\cong H(V,d)$.
\end{corollary}
\begin{proof}
By Proposition \ref{proposition:algebraic}, we have a direct sum decomposition of inner product spaces
\begin{equation*}
  V=\im dd^{\ast}\oplus \im d^{\ast}d \oplus \mathcal{H},
\end{equation*}
where $\mathcal{H}=\ker d\cap \ker d^{\ast}=\ker \Delta$. Here, $\Delta=dd^{\ast}+d^{\ast}d$. It follows that $dV=dd^{\ast}dV$. Note that $dd^{\ast}dV\subseteq dd^{\ast}V\subseteq dV$. One has $dd^{\ast}V=dV$. Thus the desired decomposition follows.

Consider the inclusion $i:\mathcal{H}\hookrightarrow V$, which gives a chain morphism $i:(\mathcal{H},0)\hookrightarrow (V,d)$. It induces an isomorphism $H(i):\mathcal{H}\cong H(\mathcal{H})\to H(V,d)$. Indeed, if $H(i)(x)=0$ for $x\in \mathcal{H}$, we have $x=i(x)=dz$ for some $z\in V$. But $\mathcal{H}\perp dV$. It follows that $x=0$ and $H(i)$ is injective.
For each $\alpha\in H(V,d)$, we can choose a representative $y\in\ker d$ of $\alpha$. By the above decomposition, we can write $y=z+dw$ for some $z\in \mathcal{H}$ and $w\in V$. Thus one has $H(i)z=\alpha$. Hence, $i$ is a quasi-isomorphism.
\end{proof}

Let $(V,d)$ be a differential graded inner product space. We define the \emph{Laplacian} $\Delta:V\to V$ as $\Delta=dd^{\ast}+d^{\ast}d$. The space $\mathcal{H}$ is called the \emph{harmonic space}. By Proposition \ref{proposition:algebraic}, we have $\ker \Delta=\mathcal{H}$.

A \emph{morphism $f:V\to W$ of differential graded inner product spaces} is a chain map such that
\begin{equation*}
  \langle x,y\rangle_{V}=\langle f(x),f(y)\rangle_{W}.
\end{equation*}
Let $\mathbf{DGI}$ be the category of differential graded inner product spaces of finite dimension over $\mathbb{R}$. If there is no ambiguity, we will always denote the differentials in $V$ and $W$ by $d$ and denote the inner products on $V$ and $W$ by $\langle\cdot,\cdot\rangle$ with no subscript.
\begin{remark}
Generally, the algebraic Hodge decomposition is not natural with respect to the morphism of differential graded inner product spaces. For example, let $V=\mathrm{span}\{a,da\}$ with the standard inner product $\langle a,a\rangle=\langle da,da\rangle=1$ and $\langle a,da\rangle=0$. Let $W=\mathrm{span}\{b,db,x\}$, $dx=0$ with inner product $\langle db,db\rangle =1$, $\langle b,b\rangle =\langle x,x\rangle= 1/2$ and $\langle b,db\rangle=\langle b,x\rangle=\langle db,x\rangle=0$. Let $f:V\to W$ be a morphism of differential graded inner product spaces given by
\begin{equation*}
  f(a)=b+x,\quad f(da)=db.
\end{equation*}
Let $p_{V},p_{W}$ be the projections of $V,W$ to the direct summands $d^{\ast}V,d^{\ast}W$, respectively. Then we have
\begin{equation*}
  p_{W}f(a+da)=p_{W}(f(a)+f(da))=p_{W}(b+x+db)=b
\end{equation*}
and
\begin{equation*}
  fp_{V}(a+da)=f(a)=b+x.
\end{equation*}
This shows that $p_{W}f\neq fp_{V}$.
\end{remark}

Our goal is to study the stability of persistent Laplacians. Hence, we hope to obtain a well-behaved Laplacian.  At least we expect the Laplacian operator to satisfy the functorial property in some sense. Another noteworthy problem is how to find a specific expression of the Hodge decomposition on each direct sum component. In other words, we want to get the decomposition at the mapping level, not just at the space level. The homotopy retract decomposition provides a way to deal with such a problem.

\begin{definition}
Let $M,N$ be chain complexes. A \emph{homotopy retract} $(M, N, i, p, K)$ of $M$ onto $N$ is a diagram of the form
\begin{equation*}
  \xymatrix{ \ar@(ul,dl)@<-5.6ex>[]_K  & M
\ar@<0.75ex>[r]^-p & N, \ar@<0.75ex>[l]^-i }
\end{equation*}
where $p,i$ are chain maps such that $pi=\mathrm{id}_{N}$ and $\mathrm{id}_{M}-ip=Kd+dK$ with $K^{2}=Ki=pK=0$.
\end{definition}

Recall that over a field, any quasi-isomorphism $i:N\to M$ extends to a homotopy retract. Now, we will carry out this process as follows for the quasi-isomorphism $i:(\mathcal{H},0)\hookrightarrow (V,d)$. Note that
\begin{equation*}
  V=dV\oplus d^{\ast}V\oplus \mathcal{H}.
\end{equation*}
Let $K:V\to V$ be a linear map given by
\begin{equation*}
  K(\mathcal{H})=K(d^{\ast}V)=0,\quad K(dd^{\ast}x)=d^{\ast}x
\end{equation*}
for any $d^{\ast}x\in d^{\ast}V$. Define the projection $p:V\to \mathcal{H}$ as follows.
\begin{equation*}
  p(dV)=p(d^{\ast}V)=0,\quad p(x)=x
\end{equation*}
for any $x\in \mathcal{H}$. Thus we have a homotopy retract  $(V, \mathcal{H}, i, p, K)$ of $V$ onto $\mathcal{H}$ of the form
\begin{equation*}
  \xymatrix{ \ar@(ul,dl)@<-5.6ex>[]_K  & V
\ar@<0.75ex>[r]^-p & \mathcal{H}. \ar@<0.75ex>[l]^-i }
\end{equation*}
It can be directly verified that $pi=\mathrm{id}_{\mathcal{H}}$ and $K^{2}=Ki=pK=0$.  On the other hand, for any element $x+dy+d^{\ast}z\in V=\mathcal{H}\oplus dV\oplus d^{\ast}V$, we have
\begin{equation*}
  x+dy+d^{\ast}z-ip(x+dy+d^{\ast}z)=dy+d^{\ast}z
\end{equation*}
and
\begin{equation*}
  Kd(x+dy+d^{\ast}z)+dK(x+dy+d^{\ast}z)=Kdd^{\ast}z+dKdy=d^{\ast}z+dy.
\end{equation*}
It follows that $\mathrm{id}-ip=Kd+dK$. Moreover, we have
\begin{proposition}\label{proposition:decomposition}
There is a decomposition on $V$ given by $\mathrm{id}=Kd+dK+h$, which gives a decomposition of $V$ as inner product spaces
\begin{equation*}
  V=dK(V)\oplus Kd(V)\oplus h(V),
\end{equation*}
where $h(V)=\mathcal{H}=\ker d\cap \ker d^{\ast}$. In addition, the decomposition coincides with the Hodge decomposition of $V$.
\end{proposition}
\begin{proof}
The quasi-isomorphism $i:(\mathcal{H},0)\hookrightarrow (V,d)$ extends to a homotopy retract $(V, \mathcal{H}, i, p, K)$ of $V$ onto $\mathcal{H}$. Let $h=ip$. We have
\begin{equation*}
  (dK)^{2}=dK,\quad (Kd)^{2}=Kd,\quad h^{2}=(ip)^{2}=ip=h.
\end{equation*}
By the direct sum decomposition theorem of linear transformations, one has
\begin{equation*}
  V=d K(V)\oplus K d (V)\oplus h(V).
\end{equation*}
By construction, we obtain $h(V)=\mathcal{H}$. It is obvious that $dK(V)\subseteq dV$. On the other hand, for any $dx\in dV$, we can write $x=dKx+Kdx+h(x)$ since $\mathrm{id}=Kd+dK+h$. Thus one has
\begin{equation*}
  dx=dKdx\in dK(V).
\end{equation*}
It follows that $dV\subseteq dKV$, and then we obtain $dV=dK(V)$. A similar process shows $d^{\ast}V=Kd(V)$. Thus the above decomposition coincides with the Hodge decomposition of $V$.
\end{proof}

Let $\mathbf{Inn}$ be the category of inner product spaces. For each differential graded inner product space, we have an inner product space $h(V)$ by the above construction. For any $f:V\to W$ of differential graded inner product spaces, we have a morphism $h(f):h(V)\to h(W),x\mapsto hf(x)$ of inner product spaces.
\begin{lemma}
The construction $h:\mathbf{DGI}\to \mathbf{Inn},V\mapsto h(V)$ is a functor from the category of differential graded inner product spaces to the category of inner product spaces.
\end{lemma}
\begin{proof}
Let $f:V\to W$ be a morphism of differential graded inner product spaces.
It is obvious that $h(\mathrm{id}|_{V})=\mathrm{id}|_{h(V)}$. For any $x\in V$, we have $dfh(x)=fdh(x)=0$. By Proposition \ref{proposition:decomposition}, one has that $fh(V)\subseteq dK(W)\oplus h(W)$. Let $g:W\to U$ be a morphism of differential graded inner product spaces. We will prove $h(g)h(f)=h(gf)$. For any $x\in h(V)$, since $f(x)\subseteq dK(W)\oplus h(W)$, we can write $f(x)=dKa+h(b)$ for some $a,b\in W$. It follows that
\begin{equation*}
  h(gf)(x)=hgf(x)=hg(dKa+h(b))=hdgKa+hgh(b)=hgh(b)
\end{equation*}
and
\begin{equation*}
  h(g)h(f)(x)=hghf(x)=hgh(dKa+h(b))=hgh^{2}(b)=hgh(b).
\end{equation*}
Thus one has $h(gf)=h(g)h(f)$, which gives the desired result.
\end{proof}
\begin{remark}
Let $f:V\to W$ be a morphism of differential graded inner product spaces.
The proof of the above lemma shows that $fh(V)\subseteq dK(W)\oplus h(W)$. Moreover, it can be verified that $f(d^{\ast}V)\subseteq d^{\ast}W\oplus h(W)$ and $fdV\subseteq dW$. For a more intuitive representation, we list these relations as follows.
\begin{equation*}
  \xymatrix{
  d^{\ast}V\ar@{->}[d]\ar@{->}[rd]&h(V)\ar@{->}[d]\ar@{->}[dr]&dV\ar@{->}[d]\\
  d^{\ast}W&h(W)&dW
  }
\end{equation*}
%\begin{equation*}
%  \xymatrix{
%  h(V)\ar@{->}[d]&d^{\ast}V\ar@{->}[d]&dV\ar@{->}[d]\\
%  h(W)\oplus dV&d^{\ast}W\oplus h(W)&dW
%  }
%\end{equation*}
\end{remark}

By Corollary \ref{corollary:decomposition}, the homology inherits the inner product from the harmonic space, which is exactly the quotient inner product.
Thus the homology also gives a functor $H:\mathbf{DGI}\to \mathbf{Inn}$ from the category of differential graded inner product spaces to the category of inner product spaces.
\begin{proposition}\label{proposition:isomorphism1}
There is a natural isomorphism $\rho:h\Rightarrow H$ of functors from $\mathbf{DGI}$ to $\mathbf{Inn}$.
\end{proposition}
\begin{proof}
Let $f:V\to W$ be a morphism of differential graded inner product spaces. For each $V\in \mathrm{ob}(\mathbf{DGI})$, define $\rho_{V}:h(V)\Rightarrow H(V)$ by $\rho_{V}(z)=[z]$. Here, $[z]$ is the homology class of $z$ in $V$. To prove $\rho$ is natural, it suffices to show $\rho_{W}\circ h(f)=H(f)\circ\rho_{V}$.
\begin{equation*}
  \xymatrix{
  h(V)\ar@{->}[r]^{\rho_{V}}\ar@{->}[d]_{h(f)}&H(V)\ar@{->}[d]^{H(f)}\\
  h(W)\ar@{->}[r]^{\rho_{W}}&H(W)
  }
\end{equation*}
For each $z\in h(V)$, let $f(z)=dKa+h(b)$ for some $a,b\in W$. Then one has
\begin{equation*}
  \rho_{W}\circ h(f)(z)=[hf(z)]=[h(b)].
\end{equation*}
On the other hand, we have
\begin{equation*}
  H(f)\circ\rho_{V}(z)=H(f)[z]=[f(z)]=[h(b)].
\end{equation*}
It follows that $\rho_{W}\circ h(f)(z)= H(f)\circ\rho_{V}(z)$. If $\rho_{V}(z)=[z]=0$, we have $z\in dV\cap h(V)=0$. So $\rho_{V}$ is injective. For any element $\alpha\in H(V)$, choose a representative $x$ of $\alpha$, that is, $\alpha=[x]$. Then we have $\rho_{V}[h(x)]=[x]$. So $\rho_{V}$ is surjective. Hence, $\rho$ is a natural isomorphism.
\end{proof}

Let $\mathcal{S}:(\mathbb{R},\leq)\to \mathbf{DGI}$ be a persistence differential graded inner product space. Then we have a persistence inner product space $h\mathcal{S}:(\mathbb{R},\leq)\to \mathbf{Inn}$. For real numbers $a\leq b$, we define the \emph{$(a,b)$-persistent harmonic space} of $\mathcal{S}$ as
\begin{equation*}
  \mathcal{H}^{a,b}(\mathcal{S})=hf_{a\to b}h(\mathcal{S}_{a}).
\end{equation*}
Here, $f_{a\to b}:\mathcal{S}_{a}\to \mathcal{S}_{b}$ is the morphism induced by $a\to b$. Recall that the homology gives a functor $H:\mathbf{DGI}\to \mathbf{Inn}$ from the category of differential graded inner product spaces to the category of inner product spaces. Therefore, we have two functors
\begin{equation*}
  \xymatrix{  (\mathbb{R},\leq)\ar@{->}[r]^{\mathcal{S}}&\mathbf{DGI}\ar@<0.75ex>[r]^-h \ar@<-0.75ex>[r]_-H& \mathbf{Inn}  }.
\end{equation*}
By Proposition \ref{proposition:isomorphism1}, one has the following theorem.
\begin{theorem}\label{theorem:nature}
Let $\mathcal{S}:(\mathbb{R},\leq)\to \mathbf{DGI}$ be a persistence differential graded inner product space.
There is a natural isomorphism $\tilde{\rho}:h\mathcal{S}\Rightarrow H\mathcal{S}$ of functors from $(\mathbb{R},\leq)$ to $\mathbf{Inn}$.
\end{theorem}

The above theorem says that the data $\{H^{a,b}(\mathcal{S})\}_{a\leq b}$ of persistent homology can be identified with the data $\{\mathcal{H}^{a,b}(\mathcal{S})\}_{a\leq b}$ of persistent harmonic spaces.

\begin{lemma}\label{lemma:section}
Let $f:V\to W$ be a morphism of differential graded inner product spaces. Then $f$ is injective. Moreover, we have $f^{\ast}f=\mathrm{id}_{V}$.
\end{lemma}
\begin{proof}
If $f(x)=0$, we have $\langle x,x\rangle_{V}=\langle f(x),f(x)\rangle_{W}=0$. By the positive definiteness of the inner product, one has $x=0$. Thus $f$ is injective. For any $v\in V$, we have
\begin{equation*}
  \langle f^{\ast}f(x)-x,v\rangle_{V}=\langle f(x),f(v)\rangle_{W}-\langle x,v\rangle_{V}=0.
\end{equation*}
It follows that $f^{\ast}f(x)-x=0$, which gives the desired result.
\end{proof}

Let $f:V\to W$ be a morphism of differential graded inner product spaces. Let $\Theta_{f}=\{x\in W|dx\in f(V)\}$. Then we have a short sequence
\begin{equation*}
\xymatrix@=1cm{& \Theta_{f}\ar@{->}[r]^{f^{\ast}d\iota}&V\ar@{->}[r]^{d}&V\ar@{->}[r]&0 }.
\end{equation*}
Here, $\iota:\Theta_{f}\hookrightarrow W$ is an inclusion. By abusing notations, we will always use the notation $\iota$ to refer to such inclusion for different $f$.
For any $x\in \Theta_{f}$, we have $dx=f(v)$ for some $v\in V$. It follows that
\begin{equation*}
  fdf^{\ast}d\iota (x)=fdf^{\ast}f(v)=fdv=df(v)=d^{2}x=0.
\end{equation*}
Since $f$ is injective, one has $df^{\ast}d\iota (x)=0$. Hence, we have $df^{\ast}d\iota =0$. We define the \emph{Laplacian of morphism $f:V\to W$} by
\begin{equation*}
  \Delta_{f}=f^{\ast}d\iota(f^{\ast}d\iota)^{\ast}+d^{\ast}d=f^{\ast}d\iota\iota^{\ast} d^{\ast}f+d^{\ast}d.
\end{equation*}
\begin{remark}
Note that all the spaces considered are graded. We use the subscript $p$ to indicate the $p$-graded component of a space.
\begin{equation*}
  \xymatrix{
  \Theta_{f,p+1}\ar@{->}[d]_{\iota_{p+1}}\ar@{->}[rr]^{(f_{p})^{\ast}d_{p+1}\iota_{p+1}}&& V_{p}\ar@{->}[d]^{f_{p}}\ar@{->}[rr]^{d_{p}}&&V_{p-1}\\
  W_{p+1}\ar@{->}[rr]^{d_{p+1}}&&W_{p}&&
  }
\end{equation*}
The \emph{Laplacian of morphism $f:V\to W$ at dimension $p$} is given by
\begin{equation*}
  \Delta_{f,p}=(f_{p})^{\ast}d_{p+1}\iota_{p+1}(\iota_{p+1})^{\ast} (d_{p+1})^{\ast}f_{p}+(d_{p})^{\ast}d_{p}.
\end{equation*}
Note that $\Delta_{f,p}:V_{p}\to V_{p}$ is a self-adjoint and non-negative definite operator.
\end{remark}
For any real numbers $a\leq b$, the \emph{$(a,b)$-persistent Laplacian} for a persistence differential graded inner product space $\mathcal{S}:(\mathbb{R},\leq)\to \mathbf{DGI}$ is defined by
\begin{equation*}
  \Delta^{a,b}_{\mathcal{S}}=f_{a\to b}^{\ast}d\iota\iota^{\ast} d^{\ast}f_{a\to b}+d^{\ast}d.
\end{equation*}
In particular, if $j_{a,b}:\mathcal{S}_{a}\hookrightarrow \mathcal{S}_{b}$ is an inclusion for any $a\leq b$, then $\mathcal{S}$ is a persistence differential graded inner product space. Thus the $(a,b)$-persistent Laplacian is $\Delta^{a,b}_{\mathcal{S}}=j_{a,b}^{\ast}d\iota\iota^{\ast}d^{\ast}j_{a,b}+d^{\ast}d$.

\begin{proposition}\label{proposition:isomorphism}
Let $f:V\to W$ be a morphism of differential graded inner product spaces.
Then we have $\ker \Delta_{f}\cong hfh(V)$.
\end{proposition}
\begin{proof}
By Proposition \ref{proposition:algebraic}, we know $V=\im f^{\ast}d\iota\iota^{\ast}d^{\ast}f\oplus \im d^{\ast}d\oplus\ker \Delta_{f}$, where $\ker\Delta_{f}=\ker d \cap\ker \iota^{\ast}d^{\ast}f$.
If $v\in \ker\Delta_{f}$, one has $\iota^{\ast}d^{\ast}f(v)=0$. It follows that
\begin{equation*}
  0=\langle \iota^{\ast}d^{\ast}f(v),z\rangle=\langle d^{\ast}f(v),\iota(z)\rangle=\langle f(v),dz\rangle
\end{equation*}
for any $z\in \Theta_{f}$. Since $v\in\ker d$, we have $v=a+db$ for $a\in h(V)$ and $b\in V$. Then
\begin{equation*}
  0=\langle f(v),df(b)\rangle=\langle f(a)+f(db),df(b)\rangle=\langle f(a),f(db)\rangle+\langle f(db),f(db)\rangle.
\end{equation*}
Note that $\langle f(a),f(db)\rangle=\langle a,db\rangle=0$. By the positive definiteness of inner product, we obtain $f(db)=0$, which implies $db=0$. Hence, we have $\ker\Delta_{f}\subseteq h(V)$.
%
%
%Note that
%\begin{equation*}
%  d^{\ast}v=d^{\ast}f^{\ast}f(v)=f^{\ast}d^{\ast}f(v)=0.
%\end{equation*}
%It follows that $v\in\ker d^{\ast}$. Hence, we have $\ker\Delta_{f}\subseteq h(V)=\ker d\cap\ker d^{\ast}$. For each element $x\in\ker \Delta_{f}\subseteq h(V)$, one has
%\begin{equation*}
%  df(x)=fdx=0,\quad d^{\ast}f(x)=0.
%\end{equation*}
%It follows that $f(x)=hf(x)\in hfh(V)$.
The map $hf:h(V)\to h(W)$ of inner product spaces restricts a map of inner product spaces
\begin{equation*}
  hf|_{\ker \Delta_{f}}:\ker \Delta_{f}\longrightarrow hfh(V),\quad x\mapsto hf(x).
\end{equation*}
If $hf|_{\ker \Delta_{f}}(v)=0$ for some $v\in\ker \Delta_{f}$, we obtain $f(v)=dx$ for some $x\in W$. It follows that $x\in \Theta_{f}$. But $\langle f(v),dz\rangle=0$ for any $z\in \Theta_{f}$. One has that $\langle f(v),f(v)\rangle=0$, which implies that $v=0$. So the map $hf|_{\ker \Delta_{f}}$ is injective.
For each $y\in hfh(V)$, there exists an element $x\in h(V)$ such that $y=hf(x)$. By Proposition \ref{proposition:algebraic}, the element $x$ can be written as
\begin{equation*}
  x=x_{1}+f^{\ast}d\iota\iota^{\ast} d^{\ast}f(b),\quad x_{1}\in \ker \Delta_{f},b\in V.
\end{equation*}
Recall that $\iota^{\ast}d^{\ast}f(b)\in \Theta_{f}$. So one has $d\iota\iota^{\ast}d^{\ast}f(b)=f(u)$ for some $u\in V$. It follows that
\begin{equation*}
  hff^{\ast}d\iota\iota^{\ast}d^{\ast}f(b)=hff^{\ast}f(u)=hf(u)=hd\iota\iota^{\ast}d^{\ast}f(b)=0.
\end{equation*}
Hence, we obtain that $hf(x_{1})=y$. Thus $hf|_{\ker \Delta_{f}}$ is surjective. So the map $hf|_{\ker \Delta_{f}}:\ker \Delta_{f}\longrightarrow hfh(V)$ is an isomorphism of inner product spaces. This completes the proof.
\end{proof}

The following result is a direct application of Proposition \ref{proposition:isomorphism}, which says that the persistent harmonic space can be identified with the kernel of the persistent Laplacian.
\begin{corollary}\label{corollary:isomorphism}
Let $\mathcal{S}:(\mathbb{R},\leq)\to \mathbf{DGI}$ be a persistence differential graded inner product space.
We have $\ker \Delta^{a,b}_{\mathcal{S}}\cong \mathcal{H}^{a,b}(\mathcal{S})$ for any $a\leq b$.
\end{corollary}

\begin{theorem}\label{theorem:persistencedecomposition}
Let $\mathcal{S}:(\mathbb{R},\leq)\to \mathbf{DGI}$ be a persistence differential graded inner product space. For any $a\leq b$, we have a direct sum decomposition of inner product spaces
\begin{equation*}
  \mathcal{S}_{a}= \ker \Delta^{a,b}_{\mathcal{S}}\oplus \im f_{a\to b}^{\ast}d\iota\iota^{\ast}d^{\ast}f_{a\to b}\oplus \im d^{\ast}d,
\end{equation*}
where $\ker \Delta^{a,b}_{\mathcal{S}}\cong\mathcal{H}^{a,b}(\mathcal{S})$.
\end{theorem}
\begin{proof}
It is a straightforward result of Proposition \ref{proposition:algebraic} and Corollary \ref{corollary:isomorphism}.
\end{proof}

\section{The category of Laplacian trees}\label{section:laplaciancategory}
In this section, we will introduce the Laplacian tree as a tool for representing and classifying persistent Laplacians. The persistence Laplacian tree plays a role in persistent Laplacian theory similar to the role played by persistence modules in persistent homology theory.
We will establish a more concrete connection between persistent homology and persistent Laplacians in Theorem \ref{theorem:equivalences}.
Our focus is not on the stability of a specific Laplacian operator, but rather on the overall behavioral stability of a family of persistent Laplacian operators. This pursuit leads to the establishment of the framework of Laplacian trees.

\begin{lemma}\label{lemma:condition}
Let $g:W\to X$ be a morphism of differential graded inner product spaces. We have a direct sum decomposition $X=g(W)\oplus X_{1}$ of inner product spaces.
Then the following conditions are equivalent
\begin{enumerate}
  \item[$(i)$] For any morphism $f:V\to W$ of differential graded inner product spaces, we have $\Delta_{f}=\Delta_{gf}$;
  \item[$(ii)$] For each $x_{1}\in X_{1}$, we have $dx_{1}=0$ or $dx_{1}=g(w)+y_{1}$ for some $w\in W$ and nonzero cycle $y_{1}\in X$;
  \item[$(iii)$] %The morphism $g:dW\to d\Theta_{g}=g(W)\cap dX$ is an isomorphism, and
$g(W)\cap dX_{1}=0$.
\end{enumerate}
\end{lemma}
\begin{proof}
$(i)\Rightarrow(ii)$. By definition, we obtain $\Delta_{f}=f^{\ast}d\iota\iota^{\ast}d^{\ast}f+d^{\ast}d$ and $\Delta_{gf}=f^{\ast}g^{\ast}d\iota\iota^{\ast}d^{\ast}gf+d^{\ast}d$. Since $\Delta_{f}=\Delta_{gf}$ for any $f$, we have $f^{\ast}d\iota\iota^{\ast}d^{\ast}f=f^{\ast}g^{\ast}d\iota\iota^{\ast}d^{\ast}gf$ for any $f$.
\begin{equation*}
  \xymatrix{
  V\ar@{->}[r]^{f} &W\ar@{->}[r]^{\iota^{\ast} d^{\ast}}\ar@{->}[d]_{g}&\Theta_{f}\ar@{->}[r]^{d\iota}&W\ar@{->}[r]^{f^{\ast}}&V\\
  &X\ar@{->}[r]^{\iota^{\ast}d^{\ast}}&\Theta_{gf}\ar@{->}[r]^{d\iota}&X\ar@{->}[u]_{g^{\ast}} &\\
 }
\end{equation*}
Take $V=W$ and $f=\mathrm{id}$. We have $dd^{\ast}=g^{\ast}d\iota\iota^{\ast}d^{\ast}g$. Let $w_{1},\dots,w_{n}$ be an orthogonal basis of $W$, and let $e_{1},\dots,e_{m}$ be an orthogonal basis of $X_{1}$. Then $g(w_{1}),\dots,g(w_{n}),e_{1},\dots,e_{m}$ is an orthogonal basis  of $X$. Let $B\in \mathbb{R}^{n\times n}$ be the representation matrix of $d$ on $W$ with respect to the chosen orthogonal basis. Then the representation matrix of $d$ on $X$ is of the form $\left(
                                                                                                                                    \begin{array}{cc}
                                                                                                                                      B' & \mathbf{0} \\
                                                                                                                                      D & C \\
                                                                                                                                    \end{array}
                                                                                                                                  \right)
$, where $B'\in \mathbb{R}^{n\times n}, C\in \mathbb{R}^{m\times m},D\in \mathbb{R}^{m\times n}$. On the other hand, the representation matrix of $g$ is $\left(
                                                             \begin{array}{cc}
                                                               I_{n} & \mathbf{0} \\
                                                             \end{array}
                                                           \right)
$. Since $dg=gd$, one has $B'=B$. Let $g(w_{1}),\dots,g(w_{n}),e'_{1},\dots,e'_{r}$ be an orthogonal basis of $\Theta_{g}$. Then the representation matrix of $\iota$ is of the form $\left(
                                                                                                                                    \begin{array}{cc}
                                                                                                                                      I_{n} & \mathbf{0} \\
                                                                                                                                      0 & Q \\
                                                                                                                                    \end{array}
                                                                                                                                  \right)
$, where $Q\in \mathbb{R}^{r\times m}$. A straightforward calculation shows that
\begin{equation*}
\begin{split}
  g^{\ast}d\iota\iota^{\ast}d^{\ast}g\left(
                      \begin{array}{c}
                        w_{1} \\
                        \vdots \\
                        w_{n} \\
                      \end{array}
                    \right)=&\left(
                                                             \begin{array}{cc}
                                                               I_{n} & \mathbf{0} \\
                                                             \end{array}
                                                           \right)\left(
                                                                                                                                    \begin{array}{cc}
                                                                                                                                      B^{T} & D^{T} \\
                                                                                                                                      \mathbf{0} & C^{T} \\
                                                                                                                                    \end{array}
                                                                                                                                  \right)\left(
                                                                                                                                    \begin{array}{cc}
                                                                                                                                      I_{n} & \mathbf{0} \\
                                                                                                                                      0 & Q^{T} \\
                                                                                                                                    \end{array}
                                                                                                                                  \right)\cdot\\
  &\left(
                                                                                                                                    \begin{array}{cc}
                                                                                                                                      I_{n} & \mathbf{0} \\
                                                                                                                                      0 & Q \\
                                                                                                                                    \end{array}
                                                                                                                                  \right)\left(
                                                                                                                                    \begin{array}{cc}
                                                                                                                                      B & \mathbf{0} \\
                                                                                                                                      D & C \\
                                                                                                                                    \end{array}
                                                                                                                                  \right)\left(
                                                                                                                                           \begin{array}{c}
                                                                                                                                             I_{n} \\
                                                                                                                                             \mathbf{0} \\
                                                                                                                                           \end{array}
                                                                                                                                         \right)\left(
                      \begin{array}{c}
                        w_{1} \\
                        \vdots \\
                        w_{n} \\
                      \end{array}
                    \right)\\
=&(B^{T}B+D^{T}Q^{T}QD)\left(
                      \begin{array}{c}
                        w_{1} \\
                        \vdots \\
                        w_{n} \\
                      \end{array}
                    \right).
\end{split}
\end{equation*}
Here, $B^{T}$ is the transpose matrix of $B$, which is the representation matrix of $d^{\ast}$. Note that
\begin{equation*}
 d d^{\ast}\left(
                      \begin{array}{c}
                        w_{1} \\
                        \vdots \\
                        w_{n} \\
                      \end{array}
                    \right)=B^{T}B\left(
                      \begin{array}{c}
                        w_{1} \\
                        \vdots \\
                        w_{n} \\
                      \end{array}
                    \right).
\end{equation*}
Since $dd^{\ast}=g^{\ast}d\iota\iota^{\ast}d^{\ast}g$, we have $D^{T}Q^{T}QD=\mathbf{0}$. It follows that $QD=\mathbf{0}$.

For each $x_{1}\in X_{1}$, we will show that $dx_{1}=g(dw_{1})$ implies $dx_{1}=0$. If $dx_{1}\neq0$, we may assume $\|w_{1}\|=\sqrt{\langle w_{1},w_{1}\rangle}=1$ and choose an orthogonal basis $w_{1},w_{2},\dots,w_{n}$ of $W$ such that $w_{2}=\lambda dw_{1}$ for $\lambda=\frac{1}{\|dw_{1}\|}\in \mathbb{R}$. Choose an orthogonal basis $\mu x_{1},e_{2}\dots,e_{m}$ of $X_{1}$ for $\mu=\frac{1}{\|x_{1}\|}\in \mathbb{R}$. Since $dx_{1}=g(dw_{1})$, we may choose an orthogonal basis $g(w_{1}),\dots,g(w_{n}),\mu x_{1},e'_{2},\dots,e'_{r}$ of $\Theta_{g}$. Then we have the representation matrices
\begin{equation*}
  D=\left(
      \begin{array}{ccccc}
        0 & \frac{\mu}{\lambda} &0 & \cdots&0 \\
        \ast & \ast &\ast &  \cdots&\ast \\
        \vdots & \vdots &\ddots&  \cdots&\ast \\
        \ast & \ast &\ast &  \cdots&\ast \\
      \end{array}
    \right),\quad Q=\left(
      \begin{array}{ccccc}
       1 & 0 &0 & \cdots&0 \\
        \ast & \ast &\ast &  \cdots&\ast \\
        \vdots & \vdots &\ddots&  \cdots&\ast \\
        \ast & \ast &\ast &  \cdots&\ast \\
      \end{array}
    \right).
\end{equation*}
Thus it is impossible that $QD=\mathbf{0}$. This leads to a contradiction. So $dx_{1}=g(dw_{1})$ implies $dx_{1}=0$.

Now, we will prove $(ii)$. For each $x_{1}\in X_{1}$, by the decomposition $X=g(W)\oplus X_{1}$, we have $dx_{1}=g(w)+y_{1}$ for some $w\in W$ and $y_{1}\in X_{1}$. If $dx_{1}\neq 0$, we will prove that
$y_{1}$ is a nonzero cycle. Suppose that $y_{1}=0$. By Corollary \ref{corollary:decomposition}, we have $w=w_{0}+dw_{1}$ for some $w_{0}\in \ker d^{\ast}$ and $dw_{1}\in \im d$. It follows that $dx_{1}=g(w_{0})+dg(w_{1})$. Hence, we have
\begin{equation*}
  g^{\ast}d\iota\iota^{\ast}d^{\ast}g(w_{0})=dd^{\ast}w_{0}=0.
\end{equation*}
By the positive definiteness of inner product, one has $\iota^{\ast}d^{\ast}g(w_{0})=0$. It follows that
\begin{equation*}
  \langle d^{\ast}g(w_{0}),z\rangle=0
\end{equation*}
for any $z\in\Theta_{g}$. Note that $g(w_{0})=dx_{1}-dg(w_{1})\in dW$. One has $g(w_{0})\in d\Theta_{g}$. Choose $dz=g(w_{0})$, we have $g(w_{0})=0$ by the positive definiteness of inner product. It follows that $w_{0}=0$. Thus we have $dx_{1}=g(dw_{1})$, which implies that $dx_{1}=0$. This leads to a contradiction.
Hence, one has $y_{1}\neq 0$. Moreover, we have $dy_{1}=-dg(w)$. This implies that $dy_{1}=0$. The condition $(ii)$ is proved.

$(ii)\Rightarrow(iii)$. %It is obvious that $g$ is injective. For any $z\in \Theta_{g}$, we have $dz=g(w)\in d\Theta_{g}$ for some $w\in W$. By the decomposition $X=g(W)\oplus X_{1}$, we have $z=g(w_{1})+x_{1}$ for some $w_{1}\in W$ and $x_{1}\in X_{1}$. Thus $dx_{1}=dz-dg(w_{1})=g(w)-g(dw_{1})\in g(W)$. By $(ii)$, we obtain that $dx_{1}=0$. It follows that $g(w)=g(dw_{1})$. Hence, the morphism $g$ is surjective. Thus $g$ is an isomorphism.
Suppose $dx_{1}\in g(W)$ for some $x_{1}\in X_{1}$. By $(ii)$, we have $dx_{1}=0$. Thus we obtain $g(W)\cap dX_{1}=0$.

$(iii)\Rightarrow(i)$. By $(iii)$, one has $\Theta_{g}=g(W)\oplus Z(X_{1})$, where $Z(X_{1})$ is the space of cycles of $X_{1}$.
For each $x\in \Theta_{gf}\subseteq \Theta_{g}$, one has $x=g(w)+x_{1}$ for some $w\in W$ and $x_{1}\in Z(X_{1})$. Since $dx\in gf(V)$, we obtain $dg(w)\in gf(V)$. Thus we have $dw=f(v)$ for some $v\in V$. We will prove
\begin{equation*}
  g\iota^{\ast}d^{\ast}f=\iota^{\ast}d^{\ast}gf.
\end{equation*}
For any $u\in V$, we have
\begin{equation*}
  \langle g\iota^{\ast}d^{\ast}f(u),g(w)+x_{1}\rangle =\langle f(u),d\iota (w)\rangle+\langle \iota^{\ast}d^{\ast}f(u),g^{\ast}(x_{1})\rangle=\langle f(u),f(v)\rangle.
\end{equation*}
On the other hand, we obtain
\begin{equation*}
  \langle \iota^{\ast}d^{\ast}gf(u),g(w)+x_{1}\rangle =\langle gf(u),d\iota g(w)\rangle+\langle gf(u),d\iota x_{1}\rangle=\langle f(u), f(v)\rangle.
\end{equation*}
It follows that $\langle g\iota^{\ast}d^{\ast}f(u),x\rangle=\langle \iota^{\ast}d^{\ast}gf(u),x\rangle$ for any $x\in \Theta_{gf}$. Thus we have $g\iota^{\ast}d^{\ast}f=\iota^{\ast}d^{\ast}gf$. Moreover, we obtain
\begin{equation*}
  f^{\ast}g^{\ast}d\iota\iota^{\ast}d^{\ast}gf=f^{\ast}d\iota g^{\ast}g\iota^{\ast}d^{\ast}f=f^{\ast}d\iota\iota^{\ast}d^{\ast}f.
\end{equation*}
The condition $(i)$ follows.
\end{proof}
\begin{corollary}
Let $g:W\to X$ be a morphism of differential graded inner product spaces. We have a direct sum decomposition $X=g(W)\oplus X_{1}$ of inner product spaces. If $dX_{1}\subseteq X_{1}$, then for any morphism $f:V\to W$ of differential graded inner product spaces, we have $\Delta_{f}=\Delta_{gf}$.
\end{corollary}

Let $f:V\to W$ be a morphism of differential graded inner product spaces. We say $f$ is \emph{split} if $df(V)\cap d(f(V)^{\bot})=0$. Here, $f(V)^{\bot}$ denotes the orthogonal complement of $f(V)$ in $W$.
\begin{theorem}\label{theorem:equivalences}
Let $\mathcal{S}:(\mathbb{R},\leq)\to \mathbf{DGI}$ be a persistence differential graded inner product space.
For real numbers $b\leq c$, the following conditions are equivalent:
\begin{itemize}
  \item[$(i)$] For any $a\leq b$, we have $\Delta_{\mathcal{S}}^{a,b}=\Delta_{\mathcal{S}}^{a,c}$ on $\mathcal{S}_{a}$;
  \item[$(ii)$] $\Delta_{\mathcal{S}}^{b}=\Delta_{\mathcal{S}}^{b,c}$ on $\mathcal{S}_{b}$;
  \item[$(iii)$] The morphism $f_{b\to c}:\mathcal{S}_{b}\to \mathcal{S}_{c}$ is split, and the map $hf_{b\to c}:\mathcal{H}^{b}(\mathcal{S})\to\mathcal{H}^{b,c}(\mathcal{S})$ induced by $b\to c$ is an isomorphism;
  \item[$(iv)$] The morphism $f_{b\to c}:\mathcal{S}_{b}\to \mathcal{S}_{c}$ is split, and for any $a\leq b$, the map $hf_{b\to c}|_{\mathcal{H}^{a,b}}:\mathcal{H}^{a,b}(\mathcal{S})\to\mathcal{H}^{a,c}(\mathcal{S})$ induced by $b\to c$ is an isomorphism;
  \item[$(v)$] The morphism $f_{b\to c}:\mathcal{S}_{b}\to \mathcal{S}_{c}$ is split, and the map $H_{\ast}(f_{b\to c}):H^{b}_{\ast}(\mathcal{S})\to H^{b,c}_{\ast}(\mathcal{S})$ induced by $b\to c$ is an isomorphism.
  \item[$(vi)$] The morphism $f_{b\to c}:\mathcal{S}_{b}\to \mathcal{S}_{c}$ is split, and for any $a\leq b$, the map $H_{\ast}(f_{b\to c})|_{H^{a,b}_{\ast}(\mathcal{S})}:H^{a,b}_{\ast}(\mathcal{S})\to H^{a,c}_{\ast}(\mathcal{S})$ induced by $b\to c$ is an isomorphism.
\end{itemize}
\end{theorem}

\begin{proof}
$(i)\Rightarrow(ii)$. It is obtained by taking $a=b$.

$(ii)\Rightarrow(iii)$. By Lemma \ref{lemma:condition}, the morphism $f_{b\to c}:\mathcal{S}_{b}\to \mathcal{S}_{c}$ is split.
By definition, $hf_{b\to c}$ is a surjection. If $hf_{b\to c}(x)=0$ for some $x\in \mathcal{H}^{b}(\mathcal{S})$, we have $f_{b\to c}(x)\in d^{\ast}\mathcal{S}_{c}\oplus d\mathcal{S}_{c}$. Note that $f_{b\to c}(\mathcal{H}^{b}(\mathcal{S}))\subseteq  \mathcal{H}^{c}(\mathcal{S})\oplus d\mathcal{S}_{c}$. It follows that $f_{b\to c}(x)\in d\mathcal{S}_{c}$. By Lemma \ref{lemma:condition}, we have a direct sum decomposition $\mathcal{S}_{c}=f_{b\to c}(\mathcal{S}_{b})\oplus X_{1}$ of inner product spaces such that $f_{b\to c}(\mathcal{S}_{b})\cap dX_{1}=0$. It follows that $f_{b\to c}(x)=df_{b\to c}(x_{0})+dx_{1}$ for some $x_{0}\in \mathcal{S}_{b},x_{1}\in X_{1}$. Thus we have
\begin{equation*}
  f_{b\to c}(x-dx_{0})=dx_{1}\in dX_{1},
\end{equation*}
which implies that $f_{b\to c}(x-dx_{0})=0$. Since $f_{b\to c}$ is injective, we have that $x=dx_{0}$. But $x\in \mathcal{H}^{b}(\mathcal{S})$, we obtain $x=0$. It follows that $hf_{b\to c}$ is an injection.

$(iii)\Rightarrow(vi)$. The morphism is given by the restriction of the isomorphism $hf_{b\to c}:\mathcal{H}^{b}(\mathcal{S})\to\mathcal{H}^{b,c}(\mathcal{S})$ induced by $b\to c$. Note that
\begin{equation*}
  \mathcal{H}^{a,c}(\mathcal{S})=hf_{a\to c}h(\mathcal{S}_{a})=h(f_{b\to c}f_{a\to b})h(\mathcal{S}_{a})=h(f_{b\to c})h(f_{a\to b})h(\mathcal{S}_{a})=h(f_{b\to c})(\mathcal{H}^{a,b}).
\end{equation*}
It follows that $hf_{b\to c}|_{\mathcal{H}^{a,b}}$ is a surjection. Thus $hf_{b\to c}|_{\mathcal{H}^{a,b}}$ is an isomorphism.

$(vi)\Rightarrow(iii)$. It is directly obtained by taking $a=b$.

$(v)\Rightarrow(i)$. Let $f_{b\to c}(\mathcal{S}_{b})^{\perp}$ be the orthogonal complement of $f_{b\to c}(\mathcal{S}_{b})$ in $\mathcal{S}_{c}$. For each element $x_{1}\in f_{b\to c}(\mathcal{S}_{b})^{\perp}$, we have $dx_{1}=f_{b\to c}(w)+y_{1}$ for some $w\in \mathcal{S}_{b}$ and $y_{1}\in f_{b\to c}(\mathcal{S}_{b})^{\perp}$. If $y_{1}=0$, we have $dx_{1}=f_{b\to c}(w)$. Since $H_{\ast}(f_{b\to c})$ is an isomorphism, one has $w=dw_{1}$ for some $w_{1}\in \mathcal{S}_{b}$. As $f_{b\to c}$ is split, we obtain that $dx_{1}=f_{b\to c}(dw_{1})=0$. If $y_{1}\neq 0$, one has $dy_{1}=-f_{b\to c}(dw)$. Since $f_{b\to c}$ is split, we have $dy_{1}=-f_{b\to c}(dw)=0$. By Lemma \ref{lemma:condition}, we have $\Delta_{f_{a\to b}}=\Delta_{f_{a\to b}f_{b\to c}}$, which shows that $\Delta_{\mathcal{S}}^{a,b}=\Delta_{\mathcal{S}}^{a,c}$ for any $a\leq b$.

In the end, we finish the proof since Theorem \ref{theorem:nature} indicates $(iii)\Leftrightarrow(v)$ and $(iv)\Leftrightarrow(vi)$.
\end{proof}

The above theorem shows the connection between the persistent Laplacian and the persistent homology. We see that if the persistent Laplacian $\Delta^{b,c}$ remains unchanged, then there is no homology generator death at parameter $c$.
\begin{corollary}\label{corollary:condition}
Let $\mathcal{S}:(\mathbb{R},\leq)\to \mathbf{DGI}$ be a persistence differential graded inner product space. Then  for real numbers $b\leq c$, $\Delta^{b,c}=\Delta^{b}$ if and only if $\beta^{b,c}=\beta^{b}$ and $f_{b\to c}:\mathcal{S}_{b}\to \mathcal{S}_{c}$ is split.
\end{corollary}

Let $V$ be a differential graded inner product space. Let $\mathcal{M}(V)$ be the set of all the morphisms of differential graded inner product spaces with domain $V$, that is, the morphisms of the form $f:V\to W$ in $\mathbf{DGI}$. Here, $\mathcal{M}(V)$ is a set since $\mathbf{DGI}$ is the category of finite dimensional differential graded inner product spaces.
The \emph{Laplacian space} $L(V)$ is defined to be the set of all the Laplacians $\Delta_{f}$ given by the morphisms $f:V\to W$ in $\mathcal{M}(V)$.
We can regard $L(V)$ as the quotient of $\mathcal{M}(V)$ by the equivalence relation given by
\begin{equation*}
  f\sim g \quad\text{if }\Delta_{f}=\Delta_{g}.
\end{equation*}
A \emph{morphism $L(\phi):L(V)\to L(X)$ of Laplacian spaces} is a morphism $\phi:V\to X$ in $\mathbf{DGI}$ such that $L(\phi)(\Delta_{f})=\Delta_{g}$, where $g$ is given by the following pushout diagram.
\begin{equation*}
  \xymatrix{
  V\ar@{->}[r]^{f}\ar@{->}[d]_{\phi}\ar@{}[dr]|{\text{\LARGE{$\ulcorner$}}}&W\ar@{->}[d]^{\psi}\\
  X\ar@{->}[r]^{g}&Y
  }
\end{equation*}
We will prove that the above definition is well defined, that is, if $\Delta_{f_{1}}=\Delta_{f_{2}}$, then one has $L(\phi)(\Delta_{f_{1}})=L(\phi)(\Delta_{f_{2}})$. Before this, we first show the following lemma, which will be used to prove Proposition \ref{proposition:definition}.
\begin{lemma}\label{lemma:helpproof}
Let $f_{1}:V\to W_{1},f_{2}:V\to W_{2}$ be morphisms of differential graded inner product spaces. If $\Delta_{f_{1}}=\Delta_{f_{2}}$, then there exists a morphism $f:V\to W$ such that $f=g_{1}f_{1}=g_{2}f_{2}$ and $\Delta_{f_{1}}=\Delta_{f_{2}}=\Delta_{f}$.
\end{lemma}
\begin{proof}
Consider the pushout of $f_{1}$ and $f_{2}$ given by the following diagram.
\begin{equation*}
  \xymatrix{
  V\ar@{->}[r]^{f_{1}}\ar@{->}[d]_{f_{2}}\ar@{}[dr]|{\text{\LARGE{$\ulcorner$}}}&W_{1}\ar@{->}[d]^{g_{1}}\\
  W_{2}\ar@{->}[r]^{g_{2}}&W
  }
\end{equation*}
Here, $W=(W_{1}\oplus W_{2})/\sim$, where $\sim$ is generated by $f_{1}(x)-f_{2}(x)$ for all $x\in V$. Note that $g_{1}:W_{1}\to W, w_{1}\mapsto w_{1}$ for $w_{1}\in W_{1}$ and $g_{2}:W_{2}\to W, w_{2}\mapsto w_{2}$ for $w_{2}\in W_{2}$.
Let $f=g_{1}f_{1}=g_{2}f_{2}$. It leaves for us to prove $\Delta_{f_{1}}=\Delta_{f}$.

Let $\tilde{V}$ be the inner product space generated by $f_{1}(x)-f_{2}(x)$ for $x\in V$.
It follows that $d\tilde{V}\subseteq \tilde{V}$ and $\tilde{V}$ is a differential graded inner product space. Moreover, the quotient space $W= (W_{1}\oplus W_{2})/\tilde{V}$ is a differential graded inner product space, which is isomorphic to the orthogonal complement $\tilde{V}^{\perp}$ of $\tilde{V}$.
We identify $\tilde{V}^{\perp}$ and $W$, and we can regard $\tilde{V}^{\perp}$ as the pushout of $f_{1}$ and $f_{2}$.
Thus we have a morphism
\begin{equation*}
  k:W\cong \tilde{V}^{\perp}\hookrightarrow W_{1}\oplus W_{2}
\end{equation*}
of differential graded inner product spaces. Let $j_{1}:W_{1}\to W_{1}\oplus W_{2}$ and $j_{2}:W_{2}\to W_{1}\oplus W_{2}$. By Corollary \ref{corollary:condition}, we have $\Delta_{j_{1}f_{1}}=\Delta_{f_{1}}=\Delta_{f_{2}}=\Delta_{j_{2}f_{2}}$.
\begin{equation*}
  \xymatrix{
  W_{1}\ar@{^{(}->}[dr]^{j_{1}}\ar@{->}[d]_{g_{1}}&\\
  W\ar@{->}[r]^-{k}&W_{1}\oplus W_{2}
  }
\end{equation*}
For any $w_{1}\in W_{1}$, we have $kg_{1}(w_{1})=k(w_{1})=w_{1}=j_{1}(w_{1})$. It follows that $kg_{1}=j_{1}$. Similarly, we obtain $kg_{2}=j_{2}$. By Corollary \ref{corollary:condition}, we have  $\Delta_{g_{1}f_{1}}=\Delta_{kg_{1}f_{1}}=\Delta_{j_{1}f_{1}}=\Delta_{f_{1}}$.
\end{proof}

\begin{proposition}\label{proposition:definition}
Let $L(\phi):L(V)\to L(X)$ be a morphism of Laplacian spaces. If $\Delta_{f_{1}}=\Delta_{f_{2}}$ in $L(V)$, then we have $L(\phi)(\Delta_{f_{1}})=L(\phi)(\Delta_{f_{2}})$.
\end{proposition}
\begin{proof}
Suppose that $f_{1}:V\to W_{1},f_{2}:V\to W_{2}$ are morphisms of differential graded inner product spaces.
By Lemma \ref{lemma:helpproof}, we can reduce to the case that $f_{2}=pf_{1}$, where $p:W_{1}\to W_{2}$ is a morphism of differential graded inner product spaces. Consider the following diagram of pushouts.
\begin{equation*}
  \xymatrix{
  V\ar@{->}[r]^{f_{1}}\ar@{->}[d]_{\phi}\ar@{}[dr]|{\text{\LARGE{$\ulcorner$}}}&W_{1}\ar@{->}[r]^{p}\ar@{->}[d]^{\psi}&W_{2}\ar@{->}[dd]^{\psi'}\\
  X\ar@{->}[r]^{g}\ar@{=}[d]_{\mathrm{id}}&Y\ar@{.>}[rd]^{q}&\\
  X\ar@{->}[rr]^{g'}&&Y'
  }
\end{equation*}
Here, $Y$ is given by the pushout of $f_{1}$ and $\phi$ while $Y'$ is given by the pushout of $pf_{1}$ and $\phi$. Hence, we have
\begin{equation*}
  Y=(W_{1}\oplus X)/\sim,
\end{equation*}
where $\sim$ is generated by $f_{1}(x)-\phi(x)$ for all $x\in V$. Besides, we obtain
\begin{equation*}
  Y'=(W_{2}\oplus X)/\sim,
\end{equation*}
where $\sim$ is generated by $pf_{1}(x)-\phi(x)$ for all $x\in V$. Consider the map
\begin{equation*}
  q:Y=(W_{1}\oplus X)/\sim\to Y'=(W_{2}\oplus X)/\sim
\end{equation*}
defined by $q(w_{1}+x)=p(w_{1})+x$ for $w_{1}\in W_{1}$ and $x\in X$. Note that $dq(w_{1}+x)=dp(w_{1})+dx=pd(w_{1})+dx=qd(w_{1}+x)$ and
\begin{equation*}
  \langle p(w_{1})+x,p(w_{1}')+x'\rangle=\langle p(w_{1}),p(w_{1}')\rangle+\langle x,x'\rangle=\langle w_{1},w_{1}'\rangle+\langle x,x'\rangle= \langle w_{1}+x,w_{1}'+x'\rangle.
\end{equation*}
It follows that $q$ is a morphism of differential graded inner product spaces. Moreover, one can verify that $q\psi=\psi'p$ and $qg=g'$ by straightforward calculations.

Since $\Delta_{f_{1}}=\Delta_{f_{2}}=\Delta_{pf_{1}}$, by Lemma \ref{lemma:condition}, we obtain a direct sum decomposition $W_{2}=p(W_{1})\oplus U$ of inner product spaces such that $p(W_{1})\cap dU=0$. Thus we have
\begin{equation*}
  Y'=(p(W_{1})\oplus U\oplus X)/\sim=(p(W_{1})\oplus X)/\sim\oplus U=q(Y)\oplus U,
\end{equation*}
where $\sim$ is generated by $pf_{1}(x)-\phi(x)$ for all $x\in V$. Note that $p(W_{1})\cap dU=X\cap dU=0$. By Lemma \ref{lemma:condition}, one has $\Delta_{g'}=\Delta_{qg}=\Delta_{g}$, which completes the proof.
\end{proof}

Proposition \ref{proposition:definition} shows that the morphism of \emph{Laplacian spaces} is well defined. Thus we define the category $\mathbf{Lap}$ of Laplacian trees with
\begin{itemize}
  \item objects: all the pairs $(V,A)$, where $A$ is a subset of $L(V)$ for $V$ in $\mathbf{DGI}$;
  \item morphisms: all the morphisms $(\phi,\Phi):(V,A)\to (W,B)$ of pairs, where $\phi:V\to W$ is a morphism of differential graded inner product spaces and $\Phi=L(\phi)|_{A}:A\to B$ is a map of sets.
\end{itemize}
Moreover, one can verify that $\mathcal{L}:\mathbf{DGI}\to \mathbf{Lap},V\mapsto (V,L(V))$ is a functor.

Let $\mathcal{S}:(\mathbb{R},\leq)\to \mathbf{DGI}$ be a persistence differential graded inner product space.
For any $a\in \mathbb{R}$, let $\mathcal{L}_{\mathcal{S}}(a)=(\mathcal{S}_{a},L^{a})$, where $L^{a}=\{\Delta^{a,t}\}_{a\leq t}$. Then for $a\leq b$, we have a morphism
\begin{equation*}
  \mathcal{L}^{a,b}_{\mathcal{S}}=(f_{a\to b},L(f_{a\to b})|_{L^{a}}): \mathcal{L}_{\mathcal{S}}(a)\to \mathcal{L}_{\mathcal{S}}(b),\quad (\mathcal{S}_{a},L^{a})\mapsto (\mathcal{S}_{b},L^{b})
\end{equation*}
given by $L(f_{a\to b})|_{L^{a}}(\Delta^{a,t})=\Delta^{b,\max(b,t)}$ for $\Delta^{a,t}\in L^{a}$. The morphism is consistent with the previous definition of morphism of Laplacian trees. Indeed, we always have a pushout diagram
\begin{equation*}
  \xymatrix{
  \mathcal{S}_{a}\ar@{->}[r]\ar@{->}[d]\ar@{}[dr]|{\text{\LARGE{$\ulcorner$}}}&\mathcal{S}_{t}\ar@{->}[d]\\
  \mathcal{S}_{b}\ar@{->}[r]&\mathcal{S}_{\max(b,t)}
  }
\end{equation*}
in the category $\mathbf{DGI}$ of differential graded inner product spaces. Note that
\begin{equation*}
  L(f_{b\to c})L(f_{a\to b})(\Delta^{a,t})=L(f_{b\to c})(\Delta^{b,\max(b,t)})=\Delta^{c,\max(c,t)}=L(f_{a\to c})(\Delta^{a,t})
\end{equation*}
for any real numbers $a\leq b\leq c$. So we have a functor
\begin{equation*}
 \mathcal{L}_{\mathcal{S}}: (\mathbb{R},\leq)\to \mathbf{Lap},\quad a\mapsto(\mathcal{S}_{a},L^{a}).
\end{equation*}
The functor $\mathcal{L}_{\mathcal{S}}: (\mathbb{R},\leq)\to \mathbf{Lap}$ is a \emph{persistence Laplacian tree}.

Furthermore, we have a functor
\begin{equation*}
  \mathscr{L}:\mathbf{DGI}^{\mathbb{R}}\to \mathbf{Lap}^{\mathbb{R}}
\end{equation*}
from the category of persistence differential graded inner product spaces to the category of persistence Laplacian trees given by $\mathscr{L}(\mathcal{S})=\mathcal{L}_{\mathcal{S}}$.
The persistence diagrams represent the information of persistent homology, while the Laplacian trees represent the information of persistent Laplacians.

\section{The stability for persistent Laplacians}\label{section:main}
The stability for persistent Laplacians is always referred to the stability for persistence Laplacian trees in our theory.
Persistence Laplacian trees play a similar role to persistence modules. Just as a persistence module collects the generators of persistent homology, a persistence Laplacian tree collects the persistent Laplacians. The stability of persistent Laplacians indicates that the overall variation of Laplacians remains stable with disturbances of a persistence object.

\subsection{The algebraic stability theorem}
We first give the algebraic stability theorem for persistence Laplacian trees on persistence differential graded inner product spaces.
\begin{theorem}\label{theorem:main2}
Let $\mathcal{S},\mathcal{T}:(\mathbb{R},\leq)\to \mathbf{DGI}$ be two persistence differential graded inner product spaces. Then the persistence Laplacian trees $\mathcal{L}_{\mathcal{S}}$ and $\mathcal{L}_{\mathcal{T}}$ are $\varepsilon$-interleaved if and only if $\mathcal{S}$ and $\mathcal{T}$ are $\varepsilon$-interleaved.
\end{theorem}
\begin{proof}
``$\Rightarrow$''. Consider the functor $\mathscr{P}:\mathbf{Lap}^{\mathbb{R}}\to \mathbf{DGI}^{\mathbb{R}}$ from the category of the persistence Laplacian trees to the category of persistence differential graded inner product spaces defined by
\begin{equation*}
  \mathscr{P}(\mathcal{I})(a)=\mathcal{I}_{a},\quad a\in \mathbb{R},
\end{equation*}
where $\mathcal{I}_{a}$ is given by $\mathcal{I}(a)=(\mathcal{I}_{a},A_{a})$. Indeed, for morphisms $\mathcal{I}\stackrel{f}{\to} \mathcal{J}\stackrel{g}{\to} \mathcal{K}$ of persistence Laplacian trees, we have
\begin{equation*}
  \mathscr{P}(gf)(a)=(gf)_{a}=g_{a}f_{a}=\mathscr{P}(g)(a)\mathscr{P}(f)(a),
\end{equation*}
where $\mathcal{I}_{a}\stackrel{f_{a}}{\to} \mathcal{J}_{a}\stackrel{g_{a}}{\to} \mathcal{K}_{a}$ are morphisms in $\mathbf{DGI}$. By \cite[Proposition 3.6]{bubenik2014categorification}, if $\mathcal{L}_{\mathcal{S}}$ and $\mathcal{L}_{\mathcal{T}}$ are $\varepsilon$-interleaved, then $\mathscr{P}\mathcal{L}_{\mathcal{S}}$ and $\mathscr{P}\mathcal{L}_{\mathcal{T}}$ are $\varepsilon$-interleaved. Note that $\mathscr{P}\mathcal{L}_{\mathcal{S}}=\mathscr{P}\mathscr{L}(\mathcal{S})=\mathcal{S}$ and $\mathscr{P}\mathcal{L}_{\mathcal{T}}=\mathcal{T}$. The desired result follows.

``$\Leftarrow$''. Suppose that $\mathcal{S}$ and $\mathcal{T}$ are $\varepsilon$-interleaved. By definition, we have two morphisms $\phi:\mathcal{S}\to \Sigma^{\varepsilon}\mathcal{T}$ and $\psi:\mathcal{T}\to \Sigma^{\varepsilon}\mathcal{S}$ in category $\mathbf{DGI}^{\mathbb{R}}$ such that the following two diagrams commute.
\begin{equation*}
  \xymatrix@=0.6cm{
  &\Sigma^{\varepsilon}\mathcal{T}\ar@{->}[rd]^{\Sigma^{\varepsilon}\psi}&\\
  \mathcal{S}\ar@{->}[ru]^{\phi}\ar@{->}[rr]^{\Sigma^{2\varepsilon}|_{\mathcal{S}}}&&\Sigma^{2\varepsilon}\mathcal{S}
  }\qquad \qquad
  \xymatrix@=0.6cm{
  &\Sigma^{\varepsilon}\mathcal{S}\ar@{->}[rd]^{\Sigma^{\varepsilon}\phi}&\\
  \mathcal{T}\ar@{->}[ru]^{\psi}\ar@{->}[rr]^{\Sigma^{2\varepsilon}|_{\mathcal{T}}}&&\Sigma^{2\varepsilon}\mathcal{T}
  }
\end{equation*}
Applying to the functor $\mathscr{L}:\mathbf{DGI}^{\mathbb{R}}\to \mathbf{Lap}^{\mathbb{R}}$ defined in Section \ref{section:laplaciancategory}, we have two commutative diagrams as follows.
%\begin{equation*}
%  \xymatrix@=0.6cm{
%  &\mathscr{L}(\Sigma^{\varepsilon}\mathcal{T})\ar@{->}[rd]^{\mathscr{L}(\Sigma^{\varepsilon}\psi)}&\\
%  \mathscr{L}(\mathcal{S})\ar@{->}[ru]^{\mathscr{L}(\phi)}\ar@{->}[rr]^{\mathscr{L}(\Sigma^{2\varepsilon})}&&\mathscr{L}(\Sigma^{2\varepsilon}\mathcal{S})
%  }\qquad \qquad
%  \xymatrix@=0.6cm{
%  &\mathscr{L}(\Sigma^{\varepsilon}\mathcal{S})\ar@{->}[rd]^{\mathscr{L}(\Sigma^{\varepsilon}\phi)}&\\
%  \mathscr{L}(\mathcal{T})\ar@{->}[ru]^{\mathscr{L}(\psi)}\ar@{->}[rr]^{\mathscr{L}(\Sigma^{2\varepsilon})}&&\mathscr{L}(\Sigma^{2\varepsilon}\mathcal{T})
%  }
%\end{equation*}
\begin{equation*}
  \xymatrix@=0.6cm{
  &\Sigma^{\varepsilon}\mathcal{L}_{\mathcal{T}}\ar@{->}[rd]^{\Sigma^{\varepsilon}\mathscr{L}(\psi)}&\\
  \mathcal{L}_{\mathcal{S}}\ar@{->}[ru]^{\mathscr{L}(\phi)}\ar@{->}[rr]^{\Sigma^{2\varepsilon}|_{\mathcal{L}_{\mathcal{S}}}}&&\Sigma^{2\varepsilon}\mathcal{L}_{\mathcal{S}}
  }\qquad \qquad
  \xymatrix@=0.6cm{
  &\Sigma^{\varepsilon}\mathcal{L}_{\mathcal{S}}\ar@{->}[rd]^{\Sigma^{\varepsilon}\mathscr{L}(\phi)}&\\
  \mathcal{L}_{\mathcal{T}}\ar@{->}[ru]^{\mathscr{L}(\psi)}\ar@{->}[rr]^{\Sigma^{2\varepsilon}|_{\mathcal{L}_{\mathcal{T}}}}&&\Sigma^{2\varepsilon}\mathcal{L}_{\mathcal{T}}
  }
\end{equation*}
Here, we do not distinguish the notation $\Sigma^{x}$ in $\mathbf{DGI}^{\mathbb{R}}$ and in $\mathbf{Lap}^{\mathbb{R}}$, and we use the fact $\mathscr{L}(\Sigma^{\varepsilon}\mathcal{S})=\mathscr{L}(\Sigma^{\varepsilon})\mathscr{L}(\mathcal{S})=\Sigma^{\varepsilon}\mathcal{L}_{\mathcal{S}}$ and $\mathscr{L}(\Sigma^{\varepsilon}\phi)=\mathscr{L}(\Sigma^{\varepsilon})\mathscr{L}(\phi)=\Sigma^{\varepsilon}\mathscr{L}(\phi)$. Thus the morphisms $\mathscr{L}(\phi):\mathcal{L}_{\mathcal{S}}\to \Sigma^{\varepsilon}\mathcal{L}_{\mathcal{T}}$ and $\mathscr{L}(\psi):\mathcal{L}_{\mathcal{T}}\to \Sigma^{\varepsilon}\mathcal{L}_{\mathcal{S}}$ gives an $\varepsilon$-interleaving between $\mathcal{L}_{\mathcal{S}}$ and $\mathcal{L}_{\mathcal{T}}$ in category $\mathbf{Lap}^{\mathbb{R}}$.
\end{proof}

Recall the definition of interleaving distance in Section \ref{subsection:interleaving}, we have the following corollary.
\begin{corollary}\label{corollary:distance}
Let $\mathcal{S},\mathcal{T}$ be two persistence differential graded inner product spaces. Then
\begin{equation*}
  d_{I}(\mathcal{L}_{\mathcal{S}},\mathcal{L}_{\mathcal{T}})=d_{I}(\mathcal{S},\mathcal{T}).
\end{equation*}
\end{corollary}
By Theorem \ref{theorem:nature}, one has $d_{I}(h(\mathcal{S}),h(\mathcal{T}))=d_{I}(H(\mathcal{S}),H(\mathcal{T}))$. Combining \cite[Proposition 3.6]{bubenik2014categorification} and Corollary \ref{corollary:distance}, we obtain
\begin{equation*}
  d_{I}(h(\mathcal{S}),h(\mathcal{T}))=d_{I}(H(\mathcal{S}),H(\mathcal{T}))\leq d_{I}(\mathcal{S},\mathcal{T})= d_{I}(\mathcal{L}_{\mathcal{S}},\mathcal{L}_{\mathcal{T}}).
\end{equation*}

\subsection{The Laplacian stability on simplicial complexes}\label{subsection:complex}
In this section, we will further study the algebraic stability for Laplacian trees on persistence simplicial complexes.

Let $K$ be a simplicial complex. We have a chain complex $C_{\ast}(K;G)$ with the boundary operator $\partial_{p}:C_{p}(K;G)\to C_{p-1}(K;G)$ given by
\begin{equation*}
  \partial_{p}[v_{0},v_{1},\dots,v_{p}]=\sum\limits_{i=0}^{p}(-1)^{i}[v_{0},\dots,\widehat{v_{i}},\dots,v_{p}],
\end{equation*}
for any simplex $[v_{0},v_{1},\dots,v_{p}]$ in $K_{p}$, where $G$ is the coefficient group and $\widehat{v_{i}}$ means that $v_{i}$ is omitted \cite{munkres2018elements}.
% Let $w:K\to \mathbb{R}^{+}$ be a positive real-valued function on $K$. We can endow $C_{\ast}(K;G)$ with an inner product given by
%\begin{equation*}
%  \langle \sigma,\tau\rangle=\left\{
%                               \begin{array}{ll}
%                                 w(\sigma), & \hbox{$\tau=\sigma$;} \\
%                                 0, & \hbox{otherwise.}
%                               \end{array}
%                             \right.
%\end{equation*}
%Here, $\sigma,\tau$ are simplices in $K$. In this section, we always take $G=\mathbb{R}$ and $w(\sigma)=1$ for each simplex $\sigma\in K$ for convenience.
Let $\mathbf{Cpx}$ be the category of simplicial complexes. We have a functor $C_{\ast}(-;\mathbb{R}):\mathbf{Cpx}\to \mathbf{Chain}_{\mathbb{R}},K \mapsto C_{\ast}(K;\mathbb{R})$ from the category of simplicial complexes to the category of chain complexes.
For each simplicial complex $K$, we endow $C_{\ast}(K;\mathbb{R})$ with an inner product given by
\begin{equation*}
  \langle \sigma,\tau\rangle=\left\{
                               \begin{array}{ll}
                                 1, & \hbox{$\tau=\sigma$;} \\
                                 0, & \hbox{otherwise.}
                               \end{array}
                             \right.
\end{equation*}
Here, $\sigma,\tau$ are simplices in $K$. It follows that $(C_{\ast}(K;\mathbb{R}),\langle\cdot,\cdot\rangle)$ is a differential graded inner product space. Let $\mathbf{Cpx}^{\hookrightarrow}$ be a subcategory of $\mathbf{Cpx}$ with simplicial complexes as objects and the inclusions of simplicial complexes as morphisms. Then the functor $C_{\ast}(-;\mathbb{R})$ induces a functor $\mathcal{C}: \mathbf{Cpx}^{\hookrightarrow}\to \mathbf{DGI}$ given by $K \mapsto  (C_{\ast}(K;\mathbb{R}),\langle\cdot,\cdot\rangle)$ with the above inner product.

\begin{example}
Let $K$ be a simplicial complex, and let $f$ be a real-valued function on $K$ such that $f(\sigma)\leq f(\tau)$ whenever $\sigma$ is a face of $\tau$. Then one has a functor $\mathcal{K}^{f}:(\mathbb{R},\leq )\to \mathbf{Cpx}^{\hookrightarrow}$  given by $\mathcal{K}^{f}(a)=f^{-1}((-\infty,a])$ for $a\in \mathbb{R}$, which is a persistence simplicial complex. It follows that the functor $\mathcal{C}\mathcal{K}^{f}:(\mathbb{R},\leq )\to \mathbf{DGI}$ is a persistence differential graded inner product space. Thus we have a Laplacian tree $\mathcal{L}^{f}=\mathcal{L}_{\mathcal{C}\mathcal{K}^{f}}$ associated to $f$.
\end{example}

Let $\mathcal{K}_{1},\mathcal{K}_{2}:(\mathbb{R},\leq )\to \mathbf{Cpx}^{\hookrightarrow}$ be two persistence simplicial complexes. If $\mathcal{K}_{1}$ and $\mathcal{K}_{2}$ are $\varepsilon$-interleaved, then $\mathcal{C}\mathcal{K}_{1}$ and $\mathcal{C}\mathcal{K}_{2}$ are $\varepsilon$-interleaved. By Theorem \ref{theorem:main2}, we have
\begin{equation*}
  d_{I}(\mathcal{L}_{\mathcal{C}\mathcal{K}_{1}},\mathcal{L}_{\mathcal{C}\mathcal{K}_{2}})=d_{I}(\mathcal{C}\mathcal{K}_{1},\mathcal{C}\mathcal{K}_{2})\leq d_{I}(\mathcal{K}_{1},\mathcal{K}_{2}).
\end{equation*}
Let $f,g$ be two non-decreasing real-valued functions on a simplicial complex $K$. Recall that $\|f-g\|_{\infty}=\sup\limits_{\sigma\in K}|f(\sigma)-g(\sigma)|$. Thus we have the algebraic stability theorem for simplicial complexes.
\begin{theorem}\label{theorem:complex}
Let $f,g$ be two non-decreasing real-valued functions on a simplicial complex $K$. Then
\begin{equation*}
%  d_{I}(h\mathcal{C}\mathcal{K}^{f},h\mathcal{C}\mathcal{K}^{f})=d_{I}(H\mathcal{C}\mathcal{K}^{f},H\mathcal{C}\mathcal{K}^{f})\leq
d_{I}(\mathcal{L}^{f},\mathcal{L}^{g})\leq \|f-g\|_{\infty}.
\end{equation*}
\end{theorem}
\begin{proof}
Let $\varepsilon=\|f-g\|_{\infty}$. Then we have inclusions $\mathcal{K}^{f}(a)\hookrightarrow \Sigma^{\varepsilon}\mathcal{K}^{g}(a)$ and $\mathcal{K}^{g}(a)\hookrightarrow \Sigma^{\varepsilon}\mathcal{K}^{f}(a)$ for any $a\in \mathbb{R}$. Consider the natural transformations $\phi:\mathcal{K}^{f}\Rightarrow \Sigma^{\varepsilon}\mathcal{K}^{g}$ and  $\psi:\mathcal{K}^{g}\Rightarrow \Sigma^{\varepsilon}\mathcal{K}^{f}$ defined by inclusions. By a direct verification, we have
\begin{equation*}
  (\Sigma^{\varepsilon}\psi)\phi=\Sigma^{2\varepsilon}|_{\mathcal{K}^{f}}: \mathcal{K}^{f}\to \Sigma^{2\varepsilon}\mathcal{K}^{f},
\end{equation*}
which makes $\mathcal{K}^{f}(a)$ into $\mathcal{K}^{f}(a+2\varepsilon)$ for any $a\in \mathbb{R}$. Similarly, we have $(\Sigma^{\varepsilon}\phi)\psi=\Sigma^{2\varepsilon}|_{\mathcal{K}^{g}}$. Thus the persistence simplicial complexes $\mathcal{K}^{f}$ and $\mathcal{K}^{g}$ are $\varepsilon$-interleaved. So we have
\begin{equation*}
  d_{I}(\mathcal{K}^{f},\mathcal{K}^{g})\leq \varepsilon.
\end{equation*}
By Theorem \ref{theorem:main2}, we obtain $d_{I}(\mathcal{L}^{f},\mathcal{L}^{g})=d_{I}(\mathcal{C}\mathcal{K}^{f},\mathcal{C}\mathcal{K}^{g})\leq d_{I}(\mathcal{K}^{f},\mathcal{K}^{g})\leq \varepsilon$.
\end{proof}

\begin{remark}
The distance between different functions defined on different simplicial complexes can also be taken into consideration. Let $(K,f)$ and $(M,g)$ be two simplicial complexes equipped with the corresponding non-decreasing real-valued functions, respectively. The persistence simplicial complexes $\mathcal{K}^{f}$ and $\mathcal{M}^{g}$ are $\varepsilon$-interleaved if there are inclusions of simplicial complexes
\begin{equation*}
  \mathcal{K}^{f}(a)\hookrightarrow \Sigma^{\varepsilon}\mathcal{M}^{g}(a),\quad\mathcal{M}^{g}(a)\hookrightarrow \Sigma^{\varepsilon}\mathcal{K}^{f}(a)
\end{equation*}
for any $a\in \mathbb{R}$. Moreover, we can obtain that $d_{I}(\mathcal{L}_{\mathcal{C}\mathcal{K}^{f}},\mathcal{L}_{\mathcal{C}\mathcal{M}^{g}})\leq d_{I}(\mathcal{K}^{f},\mathcal{M}^{g})$.
\end{remark}

\section{An application to real-valued functions on digraphs}\label{section:digraph}

Recently, the persistent path homology based on the GLMY theory \cite{grigor2012homologies,grigor2020path,grigor2017homologies} has garnered significant attention from researchers. It holds great potential for applications in handling graph-structured data. In this section, we delve into the algebraic stability theorem concerning digraphs.

\subsection{The persistent Laplacians of digraphs}
Let $G=(V,E)$ be a digraph, that is, a finite vertex set $V$ equipped with an edge set $E\subseteq V\times V$. An \emph{elementary $p$-path} on $V$ is a sequence $i_{0}i_{1}\cdots i_{p}$, where $i_{0},i_{1},\dots,i_{p}\in V$. Let $\Lambda_{p}$ be the linear space generated by all the elementary $p$-paths on $V$. We denote the basis of $\Lambda_{p}$ by $e_{i_{0}i_{1}\cdots i_{p}},i_{0},i_{1},\dots,i_{p}\in V$. Then $\Lambda_{\ast}=(\Lambda_{p})_{p}$ is a chain complex with the differential given by
\begin{equation*}
  \partial e_{i_{0}i_{1}\cdots i_{p}}=\sum\limits_{k=0}^{p}(-1)^{k}e_{i_{0}\cdots \widehat{i_{k}}\cdots i_{p}},\quad p\geq 1
\end{equation*}
and $\partial e_{i_{0}}=0$ for any $i_{0}\in V$. Here, $\widehat{i_{k}}$ means omission of the index $i_{k}$.
An \emph{allowed $p$-path} on $G$ is an elementary $p$-path $i_{0}i_{1}\cdots i_{p}$ such that $(i_{k-1},i_{k})\in E$ for $k=1,\dots,p$. Let $\mathcal{A}_{p}(G;\mathbb{R})$ be the linear space generated by all the allowed $p$-paths on $G$.
Let us endow $\mathcal{A}_{\ast}(G;\mathbb{R})$ with an inner product given by
\begin{equation*}
  \langle\sigma,\tau\rangle=\left\{
                              \begin{array}{ll}
                                1, & \hbox{$\sigma=\tau$;} \\
                                0, & \hbox{otherwise.}
                              \end{array}
                            \right.
\end{equation*}
Here, $\sigma,\tau$ are elementary paths in $\mathcal{A}_{\ast}(G;\mathbb{R})$.

Recall that the space of $\partial$-invariant $p$-paths is defined by
\begin{equation*}
  \Omega_{p}(G)=\Omega_{p}(G;\mathbb{R})=\{u\in \mathcal{A}_{p}(G;\mathbb{R})|\partial u\in \mathcal{A}_{p-1}(G;\mathbb{R})\}.
\end{equation*}
Then $\Omega_{\ast}(G)$ is a chain complex and inherits the inner product on $\mathcal{A}_{\ast}(G;\mathbb{R})$.
The \emph{path homology} of $G$ is defined by
\begin{equation*}
  H_{p}(G)=H_{p}(\Omega_{\ast}(G)),\quad p\geq 0.
\end{equation*}
The \emph{Hodge-Laplacian on digraph $G$} is defined by
\begin{equation*}
  \Delta=\partial^{\ast}\partial+\partial\partial^{\ast}.
\end{equation*}
By Corollary \ref{corollary:decomposition}, there is a direct sum decomposition
\begin{equation*}
  \Omega_{\ast}(G)=\mathcal{H}(G)\oplus\partial\Omega_{\ast}(G)\oplus \partial^{\ast}\Omega_{\ast}(G),
\end{equation*}
where $\mathcal{H}(G)\cong H_{\ast}(G)$.
%\begin{example}
%Let $\partial\Delta[2]=\{\{0,1\},\{0,2\},\{1,2\},\{0\},\{1\},\{2\}\}$ be an abstract simplicial complex. We can endow $C_{\ast}(\partial\Delta[2];\mathbb{R})$ with an inner product structure as in Section \ref{subsection:complex}. Thus there is a direct sum decomposition
%\begin{equation*}
%\begin{split}
%  C_{\ast}(\partial\Delta[2];\mathbb{R})&=\mathrm{span}\{\{0,1\}+\{1,2\}-\{0,2\},\{0\}+\{1\}+\{2\}\}\\
%    & \oplus\mathrm{span}\{\{0,1\}+\{0,2\},\{0,1\}-\{1,2\}\}\\
%    & \oplus\mathrm{span}\{\{1\}-\{0\},\{2\}-\{1\}\} .
%\end{split}
%\end{equation*}
%Here, the harmonic space $\mathcal{H}=\mathrm{span}\{\{0,1\}+\{1,2\}-\{0,2\},\{0\}+\{1\}+\{2\}\}$ is isomorphic to the homology $H_{\ast}(\partial\Delta[2];\mathbb{R})$.
%\end{example}

\begin{example}\label{example:digraph}
Let $G=(V,E)$ be a digraph with vertex set $V=\{0,1,2,3\}$ and edge set $E=\{(0,1),(1,3),(0,2),(2,3)\}$.
\begin{figure}[htbp]
\centering
\begin{tikzpicture}[scale=1.2]
\begin{scope}[thick, every node/.style={sloped,allow upside down}]
\draw (-1,0) --node {\midarrow} (0,0.4)-- node {\midarrow}(1,0);
\draw (-1,0) --node {\midarrow} (0,-0.4)-- node {\midarrow}(1,0);
\draw[fill=black!100](-1,0) circle(1pt);
\draw[fill=black!100](1,0) circle(1pt);
\draw[fill=black!100](0,0.4) circle(1pt);
\draw[fill=black!100](0,-0.4) circle(1pt);
\node [font=\fontsize{8}{6}] (node001) at (-1.2,0){$0$};
\node [font=\fontsize{8}{6}] (node001) at (1.2,0){$3$};
\node [font=\fontsize{8}{6}] (node001) at (0,0.56){$1$};
\node [font=\fontsize{8}{6}] (node001) at (0,-0.56){$2$};
\end{scope}
\end{tikzpicture}
  \caption{The digraph $G$.}
\end{figure}
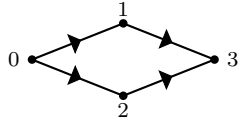
By a straightforward calculation, we have that $\Omega_{\ast}(G)=\mathrm{span}\{e_{0},e_{1},e_{2},e_{3},e_{01},e_{13},e_{02},e_{23},e_{013}-e_{023}\}.$
Moreover, there is a direct sum decomposition
\begin{equation*}
  \begin{split}
    \Omega_{\ast}(G) & = \mathrm{span}\{e_{0}+e_{1}+e_{2}+e_{3}\}\\
      & \oplus \mathrm{span}\{e_{013}-e_{023},e_{01}+e_{02},e_{01}+e_{23},e_{13}+e_{02}\}\\
      & \oplus \mathrm{span}\{e_{01}+e_{13}-e_{02}-e_{23},e_{1}-e_{0},e_{3}-e_{1},e_{2}-e_{0}\}.
  \end{split}
\end{equation*}
Here, the harmonic space $\mathcal{H}=\mathrm{span}\{e_{0}+e_{1}+e_{2}+e_{3}\}$ is isomorphic to the path homology of $G$.
\end{example}
Let $j:G_{1}\hookrightarrow  G_{2}$ be an inclusion of digraphs. Then one has $\mathcal{A}_{\ast}(G_{1})\subseteq \mathcal{A}_{\ast}(G_{2})$. It follows that $\Omega_{\ast}(G_{1})\subseteq \Omega_{\ast}(G_{2})$. Given a weighted digraph $(G,f)$, that is, a digraph $G=(V,E)$ equipped with a real-valued function $f$ on the edge set $E$, we have a functor
\begin{equation}\label{equation:persistence}
  \mathcal{D}^{f}:(\mathbb{R},\leq)\to \mathbf{Digraph}^{\hookrightarrow},\quad a\mapsto \mathcal{D}^{f}(a)=(V,f^{-1}((\infty,a])).
\end{equation}
Here, $\mathbf{Digraph}^{\hookrightarrow}$ is the subcategory of $\mathbf{Digraph}$ with digraphs as objects and the
inclusions of digraphs as morphisms. Thus we have a functor
\begin{equation*}
  \Omega_{\ast}:\mathbf{Digraph}^{\hookrightarrow}\to \mathbf{DGI},\quad G\mapsto \Omega_{\ast}(G),
\end{equation*}
which maps inclusions of digraphs to the morphisms in $\mathbf{DGI}$. Then we obtain a persistence differential graded inner product space
\begin{equation*}
  \Omega_{\ast}\mathcal{D}^{f}:(\mathbb{R},\leq)\to\mathbf{DGI}.
\end{equation*}
The \emph{$(a,b)$-persistent path homology} of the function $f$ on $G$ is
\begin{equation*}
  H^{a,b}_{\ast}(G)=\im (H_{\ast}(j_{a\to b}):H_{\ast}(\mathcal{D}^{f}(a))\to H_{\ast}(\mathcal{D}^{f}(b))).
\end{equation*}
Here, $j_{a\to b}:\Omega_{\ast}(\mathcal{D}^{f}(a))\to \Omega_{\ast}(\mathcal{D}^{f}(b))$ is a morphism of differential graded inner product spaces induced by the inclusion
$\mathcal{D}^{f}(a)\hookrightarrow \mathcal{D}^{f}(b)$ of digraphs. The \emph{$(a,b)$-persistent Laplacian} of $f$ on $G$ is
\begin{equation*}
  \Delta^{a,b}=j_{a\to b}^{\ast}d\iota\iota^{\ast}d^{\ast}j_{a\to b}\oplus d^{\ast}d,
\end{equation*}
where $\iota$ is defined as in Section \ref{section:Laplacians}. The \emph{$(a,b)$-persistent harmonic space} of $f$ on $G$ is defined by
\begin{equation*}
  \mathcal{H}^{a,b}(G)=hj_{a\to b}h\Omega_{\ast}(\mathcal{D}^{f}(a)).
\end{equation*}
We have two persistence inner product spaces $H^{f}(G)=H\Omega_{\ast}D^{f}:(\mathbb{R},\leq)\to \mathbf{Inn}$ and $\mathcal{H}^{f}(G)=h\Omega_{\ast}D^{f}:(\mathbb{R},\leq)\to \mathbf{Inn}$, which are corresponding to the persistent homology and the persistent harmonic space, respectively.
By Theorem \ref{theorem:persistencedecomposition}, we have the persistence Hodge decomposition for digraphs.
\begin{theorem}
For any real numbers $a\leq b$, there is a direct sum decomposition of inner product spaces
\begin{equation*}
  \Omega_{\ast}(\mathcal{D}^{f}(a))\cong \mathcal{H}^{a,b}(G)\oplus\im j_{a\to b}^{\ast}d\iota\iota^{\ast}d^{\ast}j_{a\to b}\oplus \im d^{\ast}d,
\end{equation*}
where $\mathcal{H}^{a,b}(G)\cong \ker \Delta^{a,b}$.
\end{theorem}
By Theorem \ref{theorem:nature}, we have a natural isomorphism $\mathcal{H}^{a,b}(G)\to H^{a,b}(G)$ with respect to the map $a\to b$. It follows that
\begin{equation*}
  \dim \mathcal{H}_{p}^{a,b}(G)=\beta^{a,b}_{p}(G),\quad a,b\in \mathbb{R},
\end{equation*}
where $\beta^{a,b}_{p}=\dim H^{a,b}_{p}(G)$ is the $(a,b)$-persistent Betti number of $f$ on $G$. The \emph{persistence Laplacian tree of $f$ on $G$} is
\begin{equation*}
  \mathcal{L}^{f}=\mathcal{L}^{(G,f)}: (\mathbb{R},\leq)\to \mathbf{Lap},\quad \mathcal{L}^{f}(a)=(\Omega_{\ast}\mathcal{D}^{f}(a),L^{a}),
\end{equation*}
where $L^{a}=\{\Delta^{a,b}\}_{a\leq b}$ is a family of Laplacians.

\begin{example}\label{example:digraph-hodge}
Example \ref{example:digraph} continued. Let $f:E\to \mathbb{R}$ be a function defined by $f(e_{01})=1,f(e_{23})=2,f(e_{13})=3,f(e_{02})=4$. We have a functor $\mathcal{D}^{f}:(\mathbb{R},\leq )\to \mathbf{Digraph}^{\hookrightarrow}$.
\begin{figure}[htbp]
\centering
\subfigure[Digraph $\mathcal{D}^{f}(1)$]{
\begin{tikzpicture}[scale=1.1]
\draw[step=.4cm,gray,ultra thin, dashed] (-1.2,-0.5) grid (1.2,0.5);
\begin{scope}[thick, every node/.style={sloped,allow upside down}]
\draw (-1,0) --node {\midarrow} (0,0.4);
\draw[fill=black!100](-1,0) circle(1pt);
\draw[fill=black!100](1,0) circle(1pt);
\draw[fill=black!100](0,0.4) circle(1pt);
\draw[fill=black!100](0,-0.4) circle(1pt);
\node [font=\fontsize{8}{6}] (node001) at (-1.2,0){$0$};
\node [font=\fontsize{8}{6}] (node001) at (1.2,0){$3$};
\node [font=\fontsize{8}{6}] (node001) at (0,0.56){$1$};
\node [font=\fontsize{8}{6}] (node001) at (0,-0.56){$2$};
\end{scope}
\end{tikzpicture}}\quad
\subfigure[Digraph $\mathcal{D}^{f}(2)$]{
\begin{tikzpicture}[scale=1.1]
\draw[step=.4cm,gray,ultra thin, dashed] (-1.2,-0.5) grid (1.2,0.5);
\begin{scope}[thick, every node/.style={sloped,allow upside down}]
\draw (-1,0) --node {\midarrow} (0,0.4);
\draw (0,-0.4)-- node {\midarrow}(1,0);
\draw[fill=black!100](-1,0) circle(1pt);
\draw[fill=black!100](1,0) circle(1pt);
\draw[fill=black!100](0,0.4) circle(1pt);
\draw[fill=black!100](0,-0.4) circle(1pt);
\node [font=\fontsize{8}{6}] (node001) at (-1.2,0){$0$};
\node [font=\fontsize{8}{6}] (node001) at (1.2,0){$3$};
\node [font=\fontsize{8}{6}] (node001) at (0,0.56){$1$};
\node [font=\fontsize{8}{6}] (node001) at (0,-0.56){$2$};
\end{scope}
\end{tikzpicture}}\quad
\subfigure[Digraph $\mathcal{D}^{f}(3)$]{
\begin{tikzpicture}[scale=1.1]
\draw[step=.4cm,gray,ultra thin, dashed] (-1.2,-0.5) grid (1.2,0.5);
\begin{scope}[thick, every node/.style={sloped,allow upside down}]
\draw (-1,0) --node {\midarrow} (0,0.4)-- node {\midarrow}(1,0);
\draw (0,-0.4)-- node {\midarrow}(1,0);
\draw[fill=black!100](-1,0) circle(1pt);
\draw[fill=black!100](1,0) circle(1pt);
\draw[fill=black!100](0,0.4) circle(1pt);
\draw[fill=black!100](0,-0.4) circle(1pt);
\node [font=\fontsize{8}{6}] (node001) at (-1.2,0){$0$};
\node [font=\fontsize{8}{6}] (node001) at (1.2,0){$3$};
\node [font=\fontsize{8}{6}] (node001) at (0,0.56){$1$};
\node [font=\fontsize{8}{6}] (node001) at (0,-0.56){$2$};
\end{scope}
\end{tikzpicture}}
\quad
\subfigure[Digraph $\mathcal{D}^{f}(4)$]{
\begin{tikzpicture}[scale=1.1]
\draw[step=.4cm,gray,ultra thin, dashed] (-1.2,-0.5) grid (1.2,0.5);
\begin{scope}[thick, every node/.style={sloped,allow upside down}]
\draw (-1,0) --node {\midarrow} (0,0.4)-- node {\midarrow}(1,0);
\draw (-1,0) --node {\midarrow} (0,-0.4)-- node {\midarrow}(1,0);
\draw[fill=black!100](-1,0) circle(1pt);
\draw[fill=black!100](1,0) circle(1pt);
\draw[fill=black!100](0,0.4) circle(1pt);
\draw[fill=black!100](0,-0.4) circle(1pt);
\node [font=\fontsize{8}{6}] (node001) at (-1.2,0){$0$};
\node [font=\fontsize{8}{6}] (node001) at (1.2,0){$3$};
\node [font=\fontsize{8}{6}] (node001) at (0,0.56){$1$};
\node [font=\fontsize{8}{6}] (node001) at (0,-0.56){$2$};
\end{scope}
\end{tikzpicture}}
\caption{The persistence digraph in Example \ref{example:digraph-hodge}.}
\end{figure}
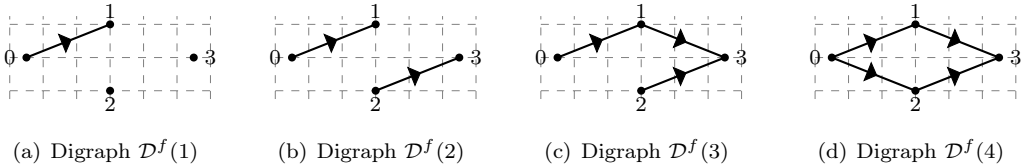
\noindent By a straightforward calculation, we have the Hodge decompositions of $\Omega_{\ast}(D^{f}(a))$ for $a=1,2,3$.
\begin{figure}[H]
\centering
\begin{tabular}{c|c|c|c}
  \hline
  $\Omega_{\ast}(\mathcal{D}^{f}(a))$ & $\mathcal{H}^{a}$ & $\im (d^{\ast})$ & $\im (d)$ \\
  \hline
  $a=1$ & $\mathrm{span}\{e_{0}+e_{1},e_{2},e_{3}\}$ & $\mathrm{span}\{e_{01}\}$ & $\mathrm{span}\{e_{1}-e_{0}\}$ \\
%  \hline
  $a=2$ & $\mathrm{span}\{e_{0}+e_{1},e_{2}+e_{3}\}$ & $\mathrm{span}\{e_{01},e_{23}\}$ & $\mathrm{span}\{e_{1}-e_{0},e_{3}-e_{2}\}$ \\
%  \hline
  $a=3$ & $\mathrm{span}\{e_{0}+e_{1}+e_{2}+e_{3}\}$ & $\mathrm{span}\{e_{01},e_{13},e_{23}\}$ & $\mathrm{span}\{e_{1}-e_{0},e_{3}-e_{1},e_{3}-e_{2}\}$ \\
  \hline
\end{tabular}
\end{figure}
\noindent Moreover, the maps of harmonic spaces are listed as follows.
\begin{equation*}
  \xymatrix@=0.3cm{
  \mathcal{H}^{1}&&\mathcal{H}^{2}&&\mathcal{H}^{3}&&\mathcal{H}^{4}\\
  e_{0}+e_{1}\ar@{->}[rr]^-{hj_{1\to 2}}&&e_{0}+e_{1}\ar@{->}[rr]^-{hj_{2\to 3}}&&\frac{1}{2}(e_{0}+e_{1}+e_{2}+e_{3})\ar@{->}[rr]^-{hj_{3\to 4}}&&\frac{1}{2}(e_{0}+e_{1}+e_{2}+e_{3})\\
  e_{2} \ar@{->}[rr]^-{hj_{1\to 2}}&&\frac{1}{2}(e_{2}+e_{3})\ar@{->}[rr]^-{hj_{2\to 3}}&&\frac{1}{4}(e_{0}+e_{1}+e_{2}+e_{3})\ar@{->}[rr]^-{hj_{3\to 4}}&&\frac{1}{4}(e_{0}+e_{1}+e_{2}+e_{3})\\
  e_{3} \ar@{->}[rr]^-{hj_{1\to 2}}&&\frac{1}{2}(e_{2}+e_{3})\ar@{->}[rr]^-{hj_{2\to 3}}&&\frac{1}{4}(e_{0}+e_{1}+e_{2}+e_{3})\ar@{->}[rr]^-{hj_{3\to 4}}&&\frac{1}{4}(e_{0}+e_{1}+e_{2}+e_{3})
  }
\end{equation*}
Here, $j_{a\to b}:\Omega_{\ast}(\mathcal{D}^{f}(a))\to \Omega_{\ast}(\mathcal{D}^{f}(b))$ is an inclusion of differential graded inner product spaces induced by the inclusion
$\mathcal{D}^{f}(a)\hookrightarrow \mathcal{D}^{f}(b)$ of digraphs. It can be obtained that
\begin{equation*}
  \mathcal{H}^{1}=\mathrm{span}\{e_{0}+e_{1},e_{2},e_{3}\},\mathcal{H}^{1,2}=\mathcal{H}^{2}=\mathrm{span}\{e_{0}+e_{1},e_{2}+e_{3}\},
\end{equation*}
\begin{equation*}
  \mathcal{H}^{1,3}=\mathcal{H}^{2,3}=\mathcal{H}^{3}=\mathcal{H}^{1,4}=\mathcal{H}^{2,4}=\mathcal{H}^{3,4}=\mathcal{H}^{4}=\mathrm{span}\{e_{0}+e_{1}+e_{2}+e_{3}\}.
\end{equation*}
The dimension of the $(a,b)$-persistent harmonic space is indeed equal to the $(a,b)$-persistent Betti number. The Laplacian trees at $a=1,2,3,4$ are listed as
\begin{equation*}
  \begin{split}
    (\Omega_{\ast}(\mathcal{D}^{f}(1)),L^{1}),& \quad L^{1}=\{\Delta^{1},\Delta^{1,2},\Delta^{1,3},\Delta^{1,4}\},\\
    (\Omega_{\ast}(\mathcal{D}^{f}(2)),L^{2}),& \quad L^{2}=\{\Delta^{2},\Delta^{2,3},\Delta^{2,4}\},\\
    (\Omega_{\ast}(\mathcal{D}^{f}(3)),L^{3}),& \quad L^{3}=\{\Delta^{3},\Delta^{3,4}\},\\
    (\Omega_{\ast}(\mathcal{D}^{f}(4)),L^{4}),& \quad L^{4}=\{\Delta^{4}\}.\\
  \end{split}
\end{equation*}
Note that $j_{3\to 4}$ is not split since $de_{02}\in j_{3\to 4}(\Omega_{\ast}(\mathcal{D}^{f}(3)))$. So we have $\Delta^{a,3}\neq\Delta^{a,4}$ for $a=1,2,3$. For example, by a step-by-step calculation, we obtain
\begin{small}
\begin{equation*}
  \Delta^{3}=\left(
                            \begin{array}{ccccccc}
                              1 & -1 & 0 & 0 & 0 & 0 & 0 \\
                              -1 & 2 & 0 & -1 & 0 & 0 & 0 \\
                              0 & 0 & 1 & -1 & 0 & 0 & 0 \\
                              0 & -1 & -1 & 2 & 0 & 0 & 0 \\
                              0 & 0 & 0 & 0 & 2 & -1 & 0 \\
                              0 & 0 & 0 & 0 & -1 & 2 & 1 \\
                              0 & 0 & 0 & 0 & 0 & 1 & 2 \\
                            \end{array}
                          \right), \Delta^{3,4}=\left(
                            \begin{array}{ccccccc}
                              2 & -1 & -1 & 0 & 0 & 0 & 0 \\
                              -1 & 2 & 0 & -1 & 0 & 0 & 0 \\
                              -1 & 0 & 2 & -1 & 0 & 0 & 0 \\
                              0 & -1 & -1 & 2 & 0 & 0 & 0 \\
                              0 & 0 & 0 & 0 & 2 & -1 & 0 \\
                              0 & 0 & 0 & 0 & -1 & 2 & 1 \\
                              0 & 0 & 0 & 0 & 0 & 1 & 2 \\
                            \end{array}
                          \right)
\end{equation*}
\end{small}
with respect to the basis $e_{0},e_{1},e_{2},e_{3},e_{01},e_{13},e_{23}$. This conforms to the assertion of Theorem \ref{theorem:equivalences}.
\end{example}

\subsection{The stability for the persistent Laplacian of digraphs}
Let $(G_{1},f)$ and $(G_{2},g)$ be weighted digraphs. We have two persistence digraphs
\begin{equation*}
  \mathcal{D}^{f}_{G_{1}},\mathcal{D}^{g}_{G_{2}}:(\mathbb{R},\leq)\to \mathbf{Digraph}^{\hookrightarrow},\quad a\mapsto \mathcal{D}^{f}(a)=(V,f^{-1}((\infty,a])).
\end{equation*}
defined by Equation (\ref{equation:persistence}). The interleaving distance of $(G_{1},f)$ and $(G_{2},g)$ is given by
\begin{equation*}
  d_{I}((G_{1},f),(G_{2},g))=d_{I}(\mathcal{D}^{f}_{G_{1}},\mathcal{D}^{g}_{G_{2}}).
\end{equation*}
Thus we have the algebraic stability theorem for digraphs.
\begin{theorem}\label{theorem:2digraph}
Let $(G_{1},f)$ and $(G_{2},g)$ be weighted digraphs. Then
\begin{equation*}
  d_{I}(\mathcal{H}^{f}(G_{1}),\mathcal{H}^{g}(G_{2}))=d_{I}(H^{f}(G_{1}),H^{g}(G_{2}))\leq d_{I}(\mathcal{L}^{(G_{1},f)},\mathcal{L}^{(G_{2},g)})\leq d_{I}((G_{1},f),(G_{2},g)).
\end{equation*}
\end{theorem}
\begin{proof}
By Theorem \ref{theorem:nature}, the first equation is established. From \cite[Proposition 3.6]{bubenik2014categorification}, one has $d_{I}(H^{f}(G),H^{g}(G))\leq d_{I}(\mathcal{L}^{f},\mathcal{L}^{g})$. By Theorem \ref{theorem:main2}, we obtain
\begin{equation*}
  d_{I}(\mathcal{L}^{(G_{1},f)},\mathcal{L}^{(G_{2},g)})=d_{I}(\Omega_{\ast}\mathcal{D}^{f}_{G_{1}},\Omega_{\ast}\mathcal{D}^{g}_{G_{2}})\leq d_{I}(\mathcal{D}^{f}_{G_{1}},\mathcal{D}^{g}_{G_{2}})=d_{I}((G_{1},f),(G_{2},g)).
\end{equation*}
The last inequality follows.
\end{proof}
For two real-valued functions $f,g$ on a digraph $G=(V,E)$, define $\|f-g\|_{\infty}=\sup\limits_{e\in E}|f(e)-g(e)|$. A corollary of Theorem \ref{theorem:2digraph} is as follows.
\begin{corollary}\label{corollary:digraph}
Let $(G,f)$ and $(G,g)$ be weighted digraphs. Then
\begin{equation*}
  d_{I}(\mathcal{H}^{f}(G),\mathcal{H}^{g}(G))=d_{I}(H^{f}(G),H^{g}(G))\leq d_{I}(\mathcal{L}^{f},\mathcal{L}^{g})\leq \|f-g\|_{\infty}.
\end{equation*}
\end{corollary}
\begin{proof}
The last inequality is obtained through a similar argument as in the proof of Theorem \ref{theorem:complex}.
\end{proof}

\medskip
\noindent {\bf Acknowledgement}.
This work is supported in part by the Natural Science Foundation of China (NSFC grant no. 11971144), High-level Scientific Research Foundation of Hebei Province and the start-up research fund from BIMSA.
\bibliographystyle{plain}
\bibliography{stability_ref}

\bigskip

\medskip

Jian Liu

Affiliation: $^1$Mathematical Science Research Center, Chongqing University of Technology, 400054, China.

$^2$Yanqi Lake Beijing Institute of Mathematical Sciences and Applications, 101408, China.

e-mail: liujian@bimsa.cn

\medskip

Jingyan Li

Affiliation: Yanqi Lake Beijing Institute of Mathematical Sciences and Applications, 101408, China.

e-mail: jingyanli@bimsa.cn

\medskip

Jie Wu

Affiliation: Yanqi Lake Beijing Institute of Mathematical Sciences and Applications, 101408, China.

e-mail: wujie@bimsa.cn

\medskip

\end{CJK*}
 \end{document}